\documentclass[12pt]{amsart}
\usepackage{amsmath, amssymb,latexsym, amsthm, amscd, mathrsfs, stmaryrd}
\usepackage[all]{xy}

\theoremstyle{definition}
\newtheorem{rem}[subsubsection]{Remark}
\theoremstyle{plain}
\newtheorem{prop}[subsubsection]{Proposition}
\newtheorem{thm}[subsubsection]{Theorem}
\newtheorem{lem}[subsubsection]{Lemma}
\newtheorem{cor}[subsubsection]{Corollary}

\newcommand{\Hom}{\mathrm{Hom}}
\newcommand{\IHom}{\mrm{IHom}}
\newcommand{\mbf}{\mathbf}
\newcommand{\mbb}{\mathbb}
\newcommand{\mrm}{\mathrm}

\newcommand{\C}{\mbb C}
\newcommand{\D}{\mathcal D}
\newcommand{\e}{\varepsilon}
\newcommand{\E}{\mathbf E}
\newcommand{\F}{\mathcal F}
\newcommand{\tF}{\widetilde{\mbf  F}}
\newcommand{\tE}{\widetilde{\mbf E}}
\newcommand{\G}{\mbf G}

\newcommand{\M}{\mathcal M}
\newcommand{\N}{\mathcal N}
\newcommand{\Q}{\mathcal Q}

\newcommand{\U}{\mbf U}
\newcommand{\V}{\mathcal  V}

\newcommand{\X}{\mbf X}

\newcommand{\wt}{\mathrm{wt}}
\newcommand{\te}{\tilde e}
\newcommand{\tf}{\tilde f}
\newcommand{\tr}{\mbox{tr}}
\newcommand{\im}{\mbox{im}}

\newcommand{\Res}{\mrm{Res}}
\newcommand{\Ind}{\mrm{Ind}}

\newcommand{\R}{\mathcal R}

\newcommand{\Irr}{\mathrm{Irr}\, }
\newcommand{\SSS}{\mrm{SS}}
\newcommand{\Supp}{\mrm{Supp}}
\newcommand{\bin}[2]{{\renewcommand\arraystretch{0.5}
\left[\!\begin{array}{c}
{\scriptstyle #1}\\{\scriptstyle #2}\end{array}\!\right]}}

\title[Tensor product varieties, perverse sheaves, and stability conditions]{Tensor Product Varieties,  Perverse Sheaves, and Stability Conditions}

\author{ Yiqiang Li}
\address{Department of Mathematics\\  University at Buffalo\\ State  University of New York\\ 244 Mathematics Building\\ Buffalo, New York 14260}

\email{ yiqiang@buffalo.edu}
\date{\today }

\begin{document}

\begin{abstract}
We show that the space spanned   by the class of simple perverse sheaves in ~\cite{Zheng08} without localization
is isomorphic to the tensor product of a Verma module with a tensor product of irreducible integrable modules of the  quantum enveloping algebra associated with  a graph. 
Under the isomorphism, the simple perverse sheaves get identified with the canonical basis elements of the tensor product module.
We also show that  the two stability conditions coincide in the localization process in ~\cite{Zheng08}, by using  supports and singular supports of complexes of sheaves, respectively. 
\end{abstract}

\maketitle

\section{Introduction}

Let $\U$ be the quantum enveloping algebra associated with a graph $\Gamma$ without loops. 
Let $M_{\lambda}$, $T_{\lambda}$ and $V_{\lambda}$ be the Verma module of $\U$ of dominant highest weight $\lambda$, its maximal submodule and its simple quotient, respectively.

In his  paper ~\cite{Zheng08}, Zheng gives a geometric realization of the tensor product $ V_{\lambda^{\bullet}}$ of irreducible integrable representations of $\U$.  
In the paper, Zheng defines  a class $\mathcal P$ of simple perverse sheaves on the frame representation varieties of an oriented graph whose underlying graph 
is  $ \Gamma$.
He then shows that, after localization, the space spanned by the class  $\mathcal P$  is isomorphic to the tensor product 
$ V_{\lambda^{\bullet}}$ and $\mathcal P$ is corresponding to the canonical basis of $V_{\lambda^{\bullet}}$.

The localization process eliminates extra elements in $\mathcal P$ in order to obtain $V_{\lambda^{\bullet}}$. 
This is similar to the process of obtaining Nakajima's quiver varieties from Lusztig's quiver varieties in \cite{Nakajima94}.
The stability condition used in the localization process in ~\cite{Zheng08}   utilizes the notion of the support of a  complex of sheaves.
On the other hand,  the notion of the singular support of  a complex of sheaves can be used  to define  the stability condition in the localization process, which is more global because it does not involve the orientation  of the given graph.  
It is not clear if  the two stability conditions for the  localization process  by using the notions of   support and singular support, respectively, are equivalent to each other.  
This question leads us to study the class $\mathcal P$ without taking the  localization process. 

It turns out that the space spanned by the class $\mathcal P$ without localization is isomorphic to the $\U$-module 
$M_0 \otimes V_{\lambda^{\bullet}}$  and, under this isomorphism, the set  $\mathcal P$ becomes 
the canonical basis (or global crystal base) of  $M_0\otimes V_{\lambda^{\bullet}}$ defined in ~\cite{Lusztig93}.
Along the way to prove the above result, we show that the singular supports of the simple perverse sheaves in $\mathcal P$ are contained in certain varieties $\mbf \Pi$ and  the irreducible components  of which have a crystal structure isomorphic to the tensor product crystal of the canonical basis of  $M_0$ with the canonical basis of 
$ V_{\lambda^{\bullet}}$. 
The quotients of the subvarieties of stable points  in $\mbf \Pi$ by certain algebraic groups are 
Malkin and Nakajima's tensor product varieties (\cite{Nakajima01, M03}).
As a  consequence,  the two stability conditions of localization are the same because they both characterize the submodule 
$T_0 \otimes V_{\lambda^{\bullet}}$. 
Our arguments also produce  a new proof of the fact that the space spanned by the class $\mathcal P$, after localization, is isomorphic to the module  $V_{\lambda^{\bullet}}$ and the elements in $\mathcal P$ survived after localization are the canonical basis elements of $V_{\lambda^{\bullet}}$.

A slight modification of the definition of the class $\mathcal P$  gives rise to a  geometric realization of the tensor product of a Verma module of dominant integrable  highest weight with 
$ V_{\lambda^{\bullet}}$ and its canonical basis. 

Note that the modules $M_0\otimes V_{\lambda^{\bullet}}$ and their variants  are projective objects in the Bernstein-Gelfand-Gelfand category $\mathcal O$  of $\U$.  
It is interesting to find a natural way to single out all indecomposable projective objects in category $\mathcal O$ from these modules.

\tableofcontents

\section{Preliminaries}

\subsection{Quantum groups}
\label{definition}

Let $\Gamma$ be a graph without loops. 
This is meant to say that we are given two finite sets $I$ and $H$, together with three maps $\; \bar \empty: H\to H$, $', '': H\to I$, such that
\[
\bar {\bar h} =h, \quad (\bar h)' =h'',\quad \mbox{and}\quad h'\neq h'',\quad \forall h\in H.
\]

\label{root}
Let $\mbf Y=\mbb Z[I]$, $\mbf X=\Hom(\mbb Z[I], \mbb Z)$, and 
$(,): \mbf Y\times \mbf X\to \mbb Z$ be the canonical pairing. 
Let $\alpha_i\in \mbf X$, for $i\in I$,  be the elements defined by 
$(i, \alpha_j)=2\delta_{ij} -\#\{ h\in H| h'=i, h''=j\}$.  The datum
$(\mbf X, \mbf Y, (,), \{i\in \mbf Y|i\in I\}, \{ \alpha_i|i\in I\})$ is the so-called simply connected root datum.

The map $i\mapsto \alpha_i$ defines an inclusion $\mbb Z[I]\hookrightarrow \mbf X$.
If $\nu\in \mbb Z[I]$, we write $\nu\in \mbf X$ to represent its image.
Let 
\[
\X^+=\{\lambda\in \X| (i, \lambda) \geq 0\quad \forall i\in I\},
\]
be the set of dominant weights.

Let $\U$ be the quantum (enveloping) algebra  associated with the above simply connected root datum. 
This is an associative algebra over $\mathbb Q(v)$, the field of rational functions over $\mbb Q$, with generators
$E_i$, $F_i$, $K_i$ and $K_i^{-1}$ for any $i \in I$ and subject to the following   relations:

\begin{equation*}
\begin{split}
(\U a).  \quad &K_0=1, \quad K_iK_{-i}=1,\quad K_iK_j=K_jK_i,\\
(\U b). \quad &K_i E_j=v^{(i,j)} E_j K_i,\quad K_i F_j=v^{-(i,j)} F_j K_i,\\
(\U c). \quad &E_iF_j-F_jE_i=\delta_{ij}\frac{ K_i-   K_{-i}}{v-v^{-1}},\\
(\U d). \quad  &\sum_{p=0}^{1-(i,j)} (-1)^p 
 \bin{1-(i,j)}{p}
E_i^p E_j E_i^{1-(i,j)-p}=0,
\quad \forall i\neq j\in I,\\
(\U e). \quad &\sum_{p=0}^{1-(i,j)} (-1)^p 
 \bin{1-(i,j)}{p}
F_i^p F_j F_i^{1-(i,j)-p}=0,
\quad \forall i\neq  j\in I,
\end{split}
\end{equation*}
where we use the following notations
\begin{equation*}
\begin{split} 
[s]=\frac{v^s-v^{-s}}{v-v^{-1}},
\quad \forall s\in \mbb Z; \quad
[s]^!=[s][s-1]\cdots [1], \quad
\bin{s}{t}=\frac{[s]^!}{[t]^![s-t]^!},
\quad \forall  s\geq t\in \mbb N.
\end{split}
\end{equation*}

The algebra $\U$ is equipped with a Hopf algebra structure whose
 comultiplication  is defined by
\begin{equation*}
\begin{split}
 \Delta(E_i)=E_i\otimes 1 +  K_i \otimes E_i, \quad 
 \Delta(F_i)=F_i \otimes K_{-i} + 1 \otimes F_i, \quad 
 \Delta(K_i)=K_i\otimes K_i,\quad \forall i\in I.
\end{split}
\end{equation*}

Let $\U^-$ be the negative part of $\U$, i.e., the subalgebra of $\U$ generated by $F_i$ for any $i\in I$.
Let $\U^+$ be the positive part of $\U$ generated by the elements $E_i$ for any $i\in I$.
Let $\U^0$ be the zero part of $\U$ generated by the elements $K_{\pm i}$ for any $i\in I$. Then we have
\[
\U =\U^+\otimes \U^0 \otimes \U^-,
\]
as vector spaces. We also set $\U^{\geq 0}=\U^0\otimes \U^+ $,  the Borel  subalgebra of $\U$.
 
Note that $\U^-$ has an $\mbb N[I]$-grading by defining $\deg (F_i)=i$ for any $i\in I$. Let $\U^-_{\nu}$ be the subspace in $\U^-$ consisting of the homogeneous elements of degree $\nu$. We have 
\[
\U^-=\oplus_{\nu\in \mbb N[I]} \U^-_{\nu}.
\]

On $\U^-$, we can define a  unique Verma module structure of $\U$ with highest weight $\lambda\in \X$ such that
\[
K_i (x)= v^{(i, \lambda-\nu)} x, \quad
F_i (x) =F_i x,\quad 
\mbox{and} \quad
E_i (1) =0, \quad \forall x\in \U^-_{\nu}, i\in I.
\]  
We write $M_{\lambda}$ for $\U^-$ together with  this $\U$-module structure.
Let $T_{\lambda}$ and $V_{\lambda}$ be the maximal submodule and simple quotient of $M_{\lambda}$.
If $\lambda\in \X^+$, then
 $T_{\lambda}$ is the left ideal of $\U^-$ generated by the elements $F_i^{(i, \lambda) +1}$ for any $i\in I$.
We have a short exact sequence of $\U$-modules:
\begin{equation}
0\to T_{\lambda} \to M_{\lambda} \to V_{\lambda} \to 0.
\end{equation}

The modules $T_{\lambda}$ and $V_{\lambda}$ are compatible with the grading on $M_{\lambda}$. In other words,
\[
T_{\lambda}=\oplus_{\nu\in \mbb N[I]} T_{\lambda, \nu}, \quad T_{\lambda,\nu} =T_{\lambda}\cap \U^-_{\nu}
\quad \mbox{and} \quad 
V_{\lambda}=\oplus_{\nu\in \mbb N[I]} V_{\lambda,\nu}, \quad V_{\lambda, \nu} =\U^-_{\nu}/T_{\lambda, \nu}.
\]

For any two $\U$-modules $M_1$ and $M_2$, we can define  on the tensor product  $M_1\otimes M_2$ a $\U$-module structure via the comultiplication of $\U$.
 Let $\mbb Q(v)_{\lambda}$ be the one dimensional $\U^{\geq 0}$ module  such that 
$E_i (f) =0$ and $K_{\pm i} f= v^{(\pm i, \lambda)} f$ for any $f\in \mbb Q(v)$ and  $i\in I$.
The Verma module $M_{\lambda}$ is the induced module $\U\otimes_{\U^{\geq 0}} \mbb Q(v)_{\lambda}$.
We have the following Frobenius  property:
\begin{equation}
\label{Frobenius}
\Hom_{\U}(M\otimes M_{\lambda}, N) \simeq \Hom_{\U^{\geq 0}} ( M\otimes \mbb Q(v)_{\lambda}, N).
\end{equation}

Let $\mbb A$ be the subring of Laurent polynomials in $\mbb Q(v)$. 
Let $F_i^{(n)}=F_i^n/[n]^!$.
The $\mbb A$-form  $_{\mbb A} \!\U^-$ of $\U^-$ is the subalgebra of $\U^-$ over $\mbb A$ generated by the elements $F_i^{(n)}$ for $i\in I$ and $n\in \mbb N$.

The $\mbb A$-form $_{\mbb A}\!\U$ of $\U$ is defined to be the subalgebra of $\U$ over $\mbb A$ generated by $F_{i}^{(n)}$, $E_i^{(n)}$ and $K_{\pm i}$ for any $i\in I$ and $n\in \mbb N$. Here $E_i^{(n)}$ is defined in a similar way as $F_i^{(n)}$.

The $\mbb A$-form $_{\mbb A}\!\U^-$ is compatible with the grading on $\U^-$, i.e., $_{\mbb A}\!\U^-=\oplus_{\nu\in \mbb N[I] }  \, _{\mbb A}\!\U^-_{\nu}$,
$_{\mbb A}\! \U^-_{\nu} =\, _{\mbb A}\! \U^- \cap \U^-_{\nu}$.

The $\mbb A$-form $_{\mbb A}\! \U^-$ is also invariant under the action of $_{\mbb A}\! \U$ on $M_{\lambda}$, denoted by $_{\mbb A}\! M_{\lambda}$.
Similarly, we can define the $\mbb A$-forms, $_{\mbb A}\! T_{\lambda}$ and $_{\mbb A} V_{\lambda}$,  of the modules $T_{\lambda}$ and $V_{\lambda}$, respectively.
Thus we have the following short exact sequence
\begin{equation}
\label{SES-A-form}
0\to\, _{\mbb A}\! T_{\lambda} \to \, _{\mbb A} \!M_{\lambda} \to  \, _{\mbb A}\!V_{\lambda} \to 0.
\end{equation}

\subsection{Crystals} 
\label{crystal}
We recall from ~\cite{KS97} the definition of a crystal. By definition, a crystal is a set $B$ equipped with five maps:
\begin{equation}
\begin{split}
&\wt : B \to \mbf X,\\
&  \e_i: B \to \mbb Z\sqcup \{ -\infty\}, \quad \varphi_i: B\to \mbb Z \sqcup \{-\infty\}, \\
&\te_i:B \to B\sqcup \{ 0\} \quad \mbox{and}\quad  \tf_i: B\to B\sqcup \{ 0\},
\end{split}
\end{equation}
and subject to the following axioms.
\begin{align*}
\tag{C1}& \varphi_i(b) =\e_i(b) + (i, \wt(b)). \\
\tag{C2} & \wt(\te_ib)= \wt(b) +\alpha_i, \; \e_i(\te_ib)=\e_i(b) -1 \;  \mbox{and}\;  \varphi_i(\te_i b ) = \varphi_i(b) +1, \quad  \mbox{if} \;  b, \te_ib \in B.\\
\tag{C2'} &
\wt(\tf_i b) =\wt(b) -\alpha_i, \; \e_i(\tf_ib) =\e_i(b) +1\;\mbox{and}\; \varphi_i(\tf_ib) = \varphi_i(b) -1, \quad \mbox{if} \; b, \tf_i b \in B.\\
\tag{C3} &
b'=\te_i b \;\mbox{if and only if}\; b=\tf_i b' , \quad  \mbox{for} \; b, b'\in B \; \mbox{and}\; i\in I.\\
\tag{C4} & 
\te_ib=\tf_i b=0,\quad  \mbox{if}\; \varphi_i (b) =-\infty.
\end{align*}
By convention, we set $\wt_i(b) =(i, \wt(b))$.

The tensor product $B_1\otimes B_2$ of the two crystals $B_1$ and $B_2$ is defined as follows. As a set, 
$B_1\otimes B_2 =\{ b_1\otimes b_2| b_1\in B_1, b_2\in B_2\}$. The five maps on $B_1\otimes B_2$ are defined by
\begin{align*}
&\wt(b_1\otimes b_2) =\wt(b_1) +\wt(b_2),\\
& \e_i(b_1\otimes b_2)=  \mbox{max} ( \e_i(b_1),  \e_i(b_2) - \wt_i(b_1) ),  \quad
 \varphi_i (b_1\otimes b_2) =\mbox{max}( \varphi_i(b_1) + \wt_i(b_2), \varphi_i(b_2)),\\
&\te_i(b_1\otimes b_2) =
\begin{cases}
\te_ib_1\otimes b_2, & \quad \varphi_i(b_1) \geq \e_i(b_2),\\
b_1\otimes \te_i b_2, & \quad  \mbox{otherwise},
\end{cases}
\;
\tf_i(b_1\otimes b_2) =
\begin{cases}
\tf_i b_1\otimes b_2, & \quad \varphi_i(b_1) > \e_i(b_2),\\
b_1\otimes \tf_ib_2, & \quad  \mbox{otherwise}.
\end{cases}
\end{align*}

A crystal morphism $\psi$ from $B_1$ to $B_2$ is a map $\psi: B_1\sqcup \{ 0\} \to B_2\sqcup \{ 0\}$  satisfying the following conditions:
\begin{align}
& \psi(0)=0,\\
& \wt(\psi( b)) = \wt(b), \; \e_i(\psi (b)) = \e_i(b) \; \mbox{and}\; \varphi_i(\psi (b)) =\varphi_i(b),\\
& \mbox{if} \; \tf_i (b) =b' \; \mbox{for} \; b, b'\in B_1\; \mbox{and}\; i\in I \; \mbox{and}\; \psi(b), \psi(b' )\in B_2, \; \mbox{then we have} \; \tf_i(\psi b) =\psi(b').
\end{align}

A $strict$ crystal homomorphism is a crystal morphism commuting with the maps $\te_i$ and $\tf_i$.

A strict crystal $isomorphism$ is a bijective strict crystal homomorphism.

We denote by $B(\lambda)$  the crystal associated with the crystal base of 
the irreducible integrable representation $V_{\lambda}$  of highest weight $\lambda$.
We also denote by $B(\lambda, \infty)$ the crystal associated with the crystal base of the Verma module $M_{\lambda}$ 
of highest weight $\lambda$.

\section{A class of $\U$-modules}

\subsection{Verma module structures on $\U^-$}
\label{Verma}

In this section, we write $\theta_i$ for the element $F_i$ in $\U^-$ to avoid confusion with the action $F_i$.
Recall that the algebra $\U^-$ has a Verma module structure  of highest weight $\lambda \in \X$ defined by
\begin{equation}
\label{defining}
K_i(x)=v^{( i, \lambda -|x| )} x,  \quad E_i(1)=0 \quad \mrm{and} \quad F_i (x)=\theta_i x,
\end{equation}
for any $x\in \U^-$ homogeneous, and $|x|$ denotes its degree. Note that the commutator relation can be rewritten as
\begin{equation}
\label{recursive}
E_i (\theta_jx)=\theta_j E_i(x) + \delta_{ij} [  ( i, \lambda -|x|) ] x, \quad \forall x\in \U^-\; \mrm{homogeneous}, i, j\in I.
\end{equation}

Recall from ~\cite{Lusztig93}, there is a unique $\mbb Q(v)$-linear map
\[
_i \bar r: \U^- \to \U^-
\]
defined by 
\[
_i \bar r(1)=0, \quad _i \bar r(\theta_j) =\delta_{ij}, \quad _i \bar r(xy)= \, _i \bar r(x) y + v^{(i, -|x|)} x \; _i \bar r(y), \quad \forall x, y\; \mbox{homogeneous}.
\]
Similarly, there is a unique $\mbb Q(v)$-linear map
\[
\bar r_i: \U^- \to \U^-
\]
defined by
\[
\bar  r_i(1)=0, \quad \bar r_i(\theta_j) =\delta_{ij}, \quad  \bar r_i(xy)= v^{(i, - |y|) } \bar r_i(x) y +  x \bar r_i(y), \quad \forall x, y\; \mbox{homogeneous}.
\]
We have

\begin{lem} 
For any  homogeneous element
$x\in \U^-$, 
\begin{equation}
\label{E}
E_i(x)=( v^{(i, \lambda)} \bar r_i(x) - v^{- ( i, \lambda - |x|+i)} \, _i\bar r(x)) / (v -v^{-1}).
\end{equation}
\end{lem}

\begin{proof}
Let $\widetilde E_i(x)$ denote the right hand side of
the equation (\ref{E}).
Then by  induction and using the definition of $_i\bar r$ and $\bar r_i$, we have
\[
[\widetilde E_i, F_j](x)=\delta_{ij}
\frac{K_i-K_{-i}}{v-v^{-1}}(x), \quad \forall x \in \U^-.
\]
From this, we have
\[
[E_i-\widetilde E_i, F_j]=0,\quad\forall i, j\in I.
\]
Observe that $(E_i-\widetilde E_i) (1)=0$. By using the fact that the
monomials in $\U^-$ span the space $\U^-$, one can show, by induction, that
$E_i(x)-\widetilde E_i(x)=0$ for all $x\in \U^-$ by using the above identity.
\end{proof}

Given any two weights $\lambda$ and $\lambda'$ in $\X$, 
we shall compare the actions $E_i$'s that define the $\U$-module  structures on $\U^-$ as $M_{\lambda}$ and $M_{\lambda+\lambda'}$, respectively. To avoid confusion, we write
$E_i^{\lambda}$  for the $E_i$ action  on $\U^-$
defining the  Verma module $M_{\lambda}$ structure. Let $\xi_{\lambda}$ denote the unit $1$ in $\U^-$ if it is regarded as the Verma module $M_{\lambda}$.

By applying (\ref{recursive}) repeatedly, 
we get
\begin{align*}
E^{\lambda}_i (\theta_{i_1} \cdots \theta_{i_m} \xi_{\lambda} )
& =\theta_{i_1} E^{\lambda}_i \theta_{i_2} \cdots \theta_{i_m}\xi_{\lambda}  
+ \delta_{i i_1} \left ( [(i , \lambda - \sum_{n=2}^m \alpha_{i_n}) ] \right ) \theta_{i_2} \cdots \theta_{i_m} \xi_{\lambda}\\
&=\sum_{k=1}^m \delta_{ii_k} 
\left ( [(i, \lambda - \sum_{n=k+1}^m \alpha_{i_n})] \right )  \theta_{i_1} \cdots \theta_{i_{k-1}} \theta_{i_{k+1}} \cdots \theta_{i_m} \xi_{\lambda}.
\end{align*}
In short, we have 
\begin{equation}
\label{quantumGLS}
E_i^{\lambda} (\theta_{i_1} \cdots \theta_{i_m} \xi_{\lambda}) 
=  \sum_{k=1}^r \delta_{i i_k}  \left (  [(i, \lambda-\sum_{n=k+1}^m \alpha_{ i_n})]  \right ) \theta_{i_1} \cdots \theta_{i_{k-1}} \theta_{i_{k+1}} \cdots \theta_{i_m}\xi_{\lambda}.
\end{equation}
Observe that
\[
[a+b] = v^b[a]+ v^{-a}[b].
\]
From this and (\ref{quantumGLS}), we have

\begin{align*}
&E_i^{\lambda+\lambda'} (\theta_{i_1} \cdots \theta_{i_m}\xi_{\lambda+\lambda'})\\
& =
v^{(i, \lambda')} E_i^{\lambda} (\theta_{i_1} \cdots \theta_{i_m}\xi_{\lambda})
+
\sum_{k=1}^m \delta_{i i_k}  
\left ( 
v^{-(i, \lambda-\sum_{m=k+1}^n \alpha_{i_n})} [(i, \lambda')] 
  \right ) \theta_{i_1} \cdots \theta_{i_{k-1}} \theta_{i_{k+1}} \cdots \theta_{i_m}\xi_{\lambda}\\
  &=v^{(i, \lambda')} E_i^{\lambda} (\theta_{i_1} \cdots \theta_{i_m}\xi_{\lambda})
  +
  v^{-(i, \lambda)}  [(i, \lambda' ) ] 
  \sum_{k=1}^m \delta_{i i_k}  
v^{(i, \sum_{n=k+1}^m\alpha_{i_n})}  \theta_{i_1} \cdots \theta_{i_{k-1}} \theta_{i_{k+1}} \cdots \theta_{i_m}\xi_{\lambda}
\end{align*}
Recall from ~\cite{Lusztig93}, there is a unique $\mbb Q(v)$-linear map
\[
r_i: \U^- \to \U^-
\]
defined by 
\[
r_i(1) =0, \quad r_i(\theta_j)=\delta_{ij}  \quad \mbox{and} \quad r_i(xy) = v^{(i, |y|)} r_i(x) y + xr_i(y).
\]
From the definition of $r_i$, we can prove by induction that 
\[
r_i ( \theta_{i_1} \cdots \theta_{i_m} ) = \sum_{k=1}^m \delta_{i i_k}  
v^{(i, \sum_{n=k+1}^m  \alpha_{i_n}) }  \theta_{i_1} \cdots \theta_{i_{k-1}} \theta_{i_{k+1}} \cdots \theta_{i_m}.
\]
So we have 
\begin{equation*}
E_i^{\lambda+\lambda'} (x\xi_{\lambda+\lambda'} ) = v^{(i, \lambda')} E_i^{\lambda} (x\xi_{\lambda} )  + v^{- (i, \lambda )} [(i, \lambda' )] r_i(x) .
\end{equation*}
 The formula  can be rewritten as 
\begin{equation}
\label{qGLS}
 E_i^{\lambda} (x\xi_{\lambda}) =v^{- (i, \lambda') } E_i^{\lambda+\lambda'} (x\xi_{\lambda+\lambda'}) - v^{- (i, \lambda+\lambda')} [(i, \lambda')]  r_i(x).
\end{equation}
Note that 
\[
r_i(x) = v^{(i, |x|-\alpha_i)} \, _i \bar r (x), \quad \forall x \; \mbox{homogeneous}.
\]
By combining (\ref{qGLS}) and this,  we have 
\begin{lem}
\label{qGLS-a}
$E_i^{\lambda} (x\xi_{\lambda}) =v^{- (i, \lambda') } E_i^{\lambda+\lambda'} (x\xi_{\lambda+\lambda'}) -
v^{- (i, \lambda+\lambda' -|x| +\alpha_i)} [(i, \lambda') ] \, _i \bar r(x) $, 
for any $\lambda$ and $\lambda'$ in $\X$ and  any homogenous element  $x\in \U^- $.
\end{lem}

This is the quantum analog of  Lemma 10 in ~\cite{GLS06}.

\subsection{The module $\mbf K'(\mbf d^{\bullet})$}
\label{raw}

Let $\tilde \Gamma$ be the framed graph of $\Gamma$. This is a graph obtained from $\Gamma$ by adding an extra copy of the vertex set, 
denoted by  $I^+=\{i^+| i\in I\}$, 
and an edge joining $i$ and  $i^+$ for each $i\in I$.

Let $\mbf K$ be the negative part $\U^-_{\tilde \Gamma} $ of the algebra $\U_{\tilde \Gamma}$  and we still write $\theta_j$ for the generators $F_j$ in $\mbf K$ when we regard it as a $\U_{\tilde \Gamma}$-module.
By Section ~\ref{Verma}, $\mbf K$ can be made into  a Verma module of $\U_{\tilde \Gamma} $ of highest weight $\lambda\equiv 0$.
Since $\Gamma$ is a subgraph of $\tilde \Gamma$, we see that $\U=\U_{\Gamma}$ is a subalgebra of $\U_{\tilde \Gamma}$. 
By restriction to $\U$, we see immediately that $\mbf K$ is a $\U$-module.

We would like to investigate the structure of  the $\U$-module $\mbf K$.
For any $\alpha \in \mathbb N[\tilde I]$, where $\tilde I$ is the vertex set of $\tilde \Gamma$,  
we write $\alpha=\alpha_I+ \alpha_{I^+}$ where $\alpha_I$ is the part supported on $I$ while $\alpha_{I^+}$ 
the part on $I^+$. 
For any $\mbf d=\sum_{i\in I} d_i i  \in \mathbb N[I]$, we set
\[
\mbf K^{\mbf d}=\{ x\in \mbf K | |x|_{I^+}= \sum_{i\in I} d_i i^{+} \}.
\]
By induction, and using ~\ref{Verma} (2), we see that $\mbf  K^{\mbf d}$ is a $\U$-module. Moreover, we have
\begin{equation}
\label{decomposition}
\mbf K=\oplus_{\mbf d  \in \mathbb Z_{\geq 0} [I]} \mbf K^{\mbf d}.
\end{equation}
For simplicity, we set
\[
\theta^{\mbf d}=\prod_{i\in I} \theta_{i^+}^{(d_i)}, \quad \forall \mbf d\in \mathbb N[I].
\]
Note that there is no ambiguity for the product since the generators $\theta_{i^+}$  commute with each other.
For a sequence $\mbf d^{\bullet}=(\mbf d^1, \cdots, \mbf d^m)$ of nonzero elements in $\mathbb N[I]$, we set
$\mbf K'(\mbf d^{\bullet})$ to be the subspace of $\mbf K$ spanned by the elements of the form
\[
x_1 \theta^{\mbf d^1} x_2 \theta^{\mbf d^2} x_3 \cdots x_m \theta^{\mbf d^m} x_{m+1}
\]
where $x_1, \cdots, x_{m+1} \in \U^-$ and regarded as elements in $\mbf K$ by the embedding $\U^-\to \U^-_{\tilde \Gamma} =\mbf K$.
Then, it is clear that $\mbf K'(\mbf d^{\bullet})$ is a $\U$-module. In fact, we only need to check elements of the above form are stable under the $E_i$ action.
This can be proved by induction on $|x|_{I}$ and using again the recursive formula ~\ref{Verma} (\ref{recursive}).
We set
\[
|\mbf d^{\bullet}|=\mbf d^1+\cdots +\mbf d^m.
\]
Now Section ~\ref{Verma} (\ref{decomposition}) can be rewritten as
\begin{equation}
\label{decomposition2}
\mbf K=\bigoplus_{\mbf d \in \mathbb N[I]} \sum_{\mbf d^{\bullet} : |\mbf d^{\bullet}|=\mbf d} \mbf K'(\mbf d^{\bullet}).
\end{equation}

Thus to study $\mbf K$, it reduces to study $\mbf K'(\mbf d^{\bullet})$.

To each $\mbf d \in \mbb N [I]$, 
we fix a $\lambda \in X^+$ such that $( i, \lambda) =d_i$ for any $i\in I$. 
Define a linear map
\[
\phi': M_0 \to \mbf K'(\mbf d)
\]
by sending $x\xi_0 \in M_0$ to  $\theta^{\mbf d} x$.
One checks immediately that the following condition (\ref{M}) holds for $\phi'$:
\begin{equation}
\label{M}
K_i \phi' (x)=v^{( i, \lambda )} \phi'(K_i x), \quad E_i \phi'(x)=\phi'(E_i x) \quad \forall x\in V_1, i\in I.
\end{equation} 
In other words, $\phi'$ is a $\U^{\geq 0}$-module homomorphism from $M_0\otimes \mbb Q(v)_{\lambda}$ to $\mbf K'(\mbf d)$.
So,  by Frobenius reciprocity (\ref{Frobenius}),  we have a unique $\U$-module homomorphism
\begin{equation}
\label{phi-1}
\phi'_{\lambda}: M_0 \otimes  M_{\lambda} \to \mbf K'(\mbf d)
\end{equation}
sending $x\xi_0\otimes \xi_{\lambda}$ to $\theta^{\mbf d}x$, for any $x\in \U^-$.
Moreover, we claim that  
\begin{equation}
\label{phi-2}
\phi'_{\lambda} (M_0\otimes T_{\lambda}) =0.
\end{equation}

 This claim can be shown by the following two steps. The first step is the observation that, with $p=(i, \lambda) +1$, 
\[
\phi'_{\lambda} ( y \xi_0 \otimes \theta_i^{(m)} \xi_{\lambda}) 
=\left ( \sum_{t=0}^m (-1)^t v^{-t(p-m)} \theta_i^{(m-t)} \theta^{\mbf d} \theta_i^{(t)} \right  ) y, \quad y\in \U^-, m\in \mbb N.
\]
This identity can be proved by induction on $m$. 
If $m= p$, it simplifies to 
\begin{equation*}
\begin{split}
\phi'_{\lambda} ( y \xi_0 \otimes \theta_i^{(m)} \xi_{\lambda}) 
=\left ( \sum_{t=0}^p (-1)^t  \theta_i^{(p-t)} \theta^{\mbf d}  \theta_i^{(t)} \right  ) y  
=\left ( \sum_{t=0}^p (-1)^t  \theta_i^{(p-t)}  \theta_{i^+}^{(d_i)} \theta_i^{(t)} \right  ) \left (\prod_{j\neq i} \theta_{j^+}^{(d_j)}  \right ) y.
\end{split}
\end{equation*}
The term $\ \sum_{t=0}^p (-1)^t  \theta_i^{(p-t)}  \theta_{i^+}^{(d_i)} \theta_i^{(t)} $ is equal to zero because it  is exactly the higher order quantum Serre relation 
$f_{i, j; n, m, e}$  with $i=i$, $j=i^+$ and $m=n+1=p$ in the algebra $\mbf f$ attached to the graph $\tilde \Gamma$
in ~\cite[7.1.1]{Lusztig93}. So we have
\[
\phi'_{\lambda} ( y \xi_0 \otimes \theta_i^{( d_i+1)} \xi_{\lambda}) =0, \quad \forall y\in \U^-.
\]
The second step is to observe that 
\[
\phi'_{\lambda} ( y \xi_0 \otimes x \theta_i^{( d_i +1)} \xi_{\lambda}) =0 , \quad \forall x, y\in \U^-,
\]
which can be proved by induction on the degree of $x$.  The claim  follows.

Now, from (\ref{phi-1}) and (\ref{phi-2}),  we see that $\phi'_{\lambda}$ factors through the module $M_0\otimes V_{\lambda}$. In particular, we have a surjective  homomorphism of $\U$ modules
\begin{equation}
\label{psi'}
\psi'_{\lambda}: M_0\otimes V_{\lambda} \twoheadrightarrow \mbf K'(\mbf d).
\end{equation}

Let $\mbf T'(\mbf d)$ be the submodule of $\mbf K'(\mbf d)$ generated by the elements  $x_1 \theta^{\mbf d} x_2 $ such that $x_2$ 
is homogeneous and of degree $ \not = 0$. In other words,
$\mbf T'(\mbf d)=\sum_{i\in I} y_i \theta^{\mbf d} x_i \theta_i$, where $x_i, y_i\in \U^-$.
Then the restriction of $\psi' _{\lambda}$ to the submodule $T_0\otimes V_{\lambda}$ has image $\mbf T'(\mbf d)$.
Thus we have the following commutative diagram:
\[
\begin{CD}
0 @>>> T_0\otimes V_{\lambda} @>>> M_0\otimes V_{\lambda} @>>> V_0\otimes V_{\lambda}=V_{\lambda} @>>> 0\\
@. @VVV  @V\psi'_{\lambda} VV  @V\bar \psi'_{\lambda} VV @.\\
0 @>>> \mbf T'(\mbf d) @>>> \mbf K'(\mbf d) @>>> \mbf V'(\mbf d) @>>> 0,
\end{CD}
\]
where we set
\[
\mbf V'(\mbf d) =\mbf K'(\mbf d) / \mbf T'(\mbf d).
\]
From the fact that $\psi'_{\lambda}$ is surjective, we see that $\bar \psi'_{\lambda}$ is surjective:
\begin{equation}
\label{barpsi'}
\bar \psi'_{\lambda}: V_{\lambda} \twoheadrightarrow \mbf V'(\mbf d).
\end{equation}
Since $V_{\lambda}$ is simple and  the generator $\theta^{\mbf d}$ is not in $\mbf T'(\mbf d)$, we see that $\bar \psi'_{\lambda}$ is an isomorphism.

We now generalize (\ref{psi'}) and (\ref{barpsi'})  to arbitrary cases.

For $\mbf d^{\bullet}=(\mbf d^1, \cdots, \mbf d^N)$, we set $\mbf d^{\bullet -1}=(\mbf d^2, \cdots \mbf d^N)$.
Let $\lambda^l$ be the chosen element in $\mbb N[I]$ such that  $( i, \lambda^l )  =d_i^l$ for all $l=1,\cdots, N$.
Let $\lambda^{\bullet}=(\lambda^N, \cdots, \lambda^1)$.  Note that $\lambda^{\bullet}$ is in the reverse order of $\mbf d^{\bullet}$.
We set 
\[
V_{\lambda^{\bullet}}=V_{\lambda^N} \otimes \cdots \otimes V_{\lambda^1}.
\]
Define a linear map 
\[
\phi'_{\lambda^{\bullet}}: \mbf K'( \mbf d^{\bullet -1}) \to \mbf K'(\mbf d^{\bullet})
\]
by 
$\phi'_{\lambda^{\bullet}} (x)= \theta^{\mbf d^1} x$,  for any  $x\in \mbf K'(\mbf d^{\bullet -1})$.
One checks again that the condition (\ref{M}) holds. So we have a unique $\U$-module homomorphism
\[
\phi'_{\lambda^{\bullet}}:   \mbf K'(\mbf d^{\bullet -1}) \otimes M_{\lambda^1}   \to \mbf K'(\mbf d^{\bullet}).
\]
sending $ x \otimes \xi_{\mbf d^1}$ to $\theta^{\mbf d^1}x$. Moreover, we may again prove that 
$\phi'( \mbf K'(\mbf d^{\bullet-1})\otimes T_{\lambda^1} )=0$ in exactly the same manner as the proof of (\ref{phi-2}).  
So $\phi'_{\lambda^{\bullet}}$ factors through the module  $\mbf K'(\mbf d^{\bullet -1}) \otimes V_{\lambda^1}  $. 
In particular, we have a surjective algebra homomorphism
\[
\psi'_{\lambda^{\bullet}}:  \mbf K'(\mbf d^{\bullet -1}) \otimes V_{\lambda^1}   \to \mbf K'(\mbf d^{\bullet}).
\]
Let $\mbf T'(\mbf d^{\bullet})$ be the submodule of $\mbf K'(\mbf d^{\bullet}) $ spanned by the elements 
$x_1 \theta^{\mbf d^1} x_2 \theta^{\mbf d^2} x_3 \cdots x_N \theta^{\mbf d^N} x_{N+1} $ 
such that is $x_{N+1}$ homogeneous and of degree $>0$. Let
\[
\mbf V'(\mbf d^{\bullet}) = \mbf K'(\mbf d^{\bullet}) / \mbf T'(\mbf d^{\bullet}), 
\]
be the quotient module.
By induction, the  homomorphism 
$V_{\mbf \lambda^{\bullet-1}}\twoheadrightarrow \mbf K'(\mbf d^{\bullet-1})/\mbf T'(\mbf d^{\bullet-1})$ is surjective. 
Moreover,
$\psi'_{\lambda^{\bullet}} ( \mbf T'(\mbf d^{\bullet-1})\otimes V_{\lambda^1} ) = \mbf T'(\mbf d^{\bullet})$. 
So 
 the morphism $\psi'_{\lambda^{\bullet}}$ induces a surjective  homomorphism of $\mbf U$ modules:
\[
\bar \psi'_{\lambda^{\bullet}}: V_{\lambda^{\bullet}} \twoheadrightarrow \mbf V'(\mbf d^{\bullet}).
\]

Summing up, we have the following proposition.

\begin{prop}
There is a  unique surjective $\U$-module homomorphism 
\[
\psi'_{\lambda^{\bullet}}: M_0 \otimes V_{\lambda^{\bullet}} \twoheadrightarrow \mbf K'(\mbf d^{\bullet}),
\]
sending $x\xi_0\otimes \xi_{\lambda^N} \otimes \cdots \otimes \xi_{\lambda^1}$ to $\theta^{\mbf d^1} \theta^{\mbf d^2} \cdots \theta^{\mbf d^N}x$ for any $x\in \U^-$. Moreover, it induces a surjective  homomorphism of $\U$-modules
$\bar \psi'_{\lambda^{\bullet}}: V_{\lambda^{\bullet}} \twoheadrightarrow \mbf V'(\mbf d^{\bullet})$.
\end{prop}

\subsection{The module $\mbf K(\mbf d^{\bullet})$}

Let $\mbf K(\mbf d^{\bullet})$ be the subspace of $\mbf K'(\mbf d^{\bullet})$ spanned by the elements of the form
\[
x_1 \theta^{\mbf d^1} x_2 \theta^{\mbf d^2} x_3 \cdots x_N \theta^{\mbf d^N}
\]
for any $x_1, \cdots, x_N \in \U^-$.  

Note that by taking $\mbf d^N=0$, we have  $\mbf K(\mbf d^{\bullet}) = \mbf K'(\mbf d^{\bullet})$ with $\mbf d^{\bullet}=(\mbf d^1,\cdots, \mbf d^{N-1})$. 

By a similar argument as in Section ~\ref{raw}, we have a unique surjective $\U$-homomorphism
\begin{equation}
\label{psi-1}
\psi_{\lambda^{\bullet}}: M_{\lambda^N} \otimes V_{\lambda^{N-1}}\otimes \cdots \otimes V_{\lambda^1}  \twoheadrightarrow 
 \mbf K(\mbf d^{\bullet}),
\end{equation}
sending $x\xi_{\lambda^N} \otimes \cdots \otimes \xi_{\lambda^1}$ to $\theta^{\mbf d^1} \cdots x\theta^{\mbf d^N}$ for any $x\in \U^-$.

Let $\mbf T(\mbf d^{\bullet})$ be the submodule of $\mbf K(\mbf d^{\bullet}) $ spanned by the elements 
$x_1 \theta^{\mbf d^1} x_2 \theta^{\mbf d^2} x_3 \cdots x_N \theta_i^{d_i+1} \theta^{\mbf d^N}$ 
for any $i\in I$ and $x_1,\cdots, x_N\in \U^-$.  Let
\[
\mbf V(\mbf d^{\bullet}) = \mbf K(\mbf d^{\bullet}) / \mbf T(\mbf d^{\bullet}), 
\]
be the quotient module. Since $\psi_{\lambda^{\bullet}} (T_{\lambda^N}\otimes V_{\lambda^{N-1}}\otimes \cdots \otimes V_{\lambda^1}) \subseteq \mbf T(\mbf d^{\bullet})$, we see that $\psi_{\lambda^{\bullet}}$ factors through the following surjective morphism
\[
\bar \psi_{\lambda^{\bullet}}: V_{\lambda^{\bullet}} \twoheadrightarrow \mbf V(\mbf d^{\bullet}).
\]

\begin{rem}
 (1). When $N=1$,   the isomorphism $\psi_{\lambda }: M_{\lambda} \to \mbf K(\mbf d^{\bullet})$ is obvious.

(2).  When $N=2$,  the morphism 
 $\psi'_{\lambda}: M_0\otimes V_{\lambda^1} \to \mbf K(\mbf d^{\bullet})$ is an isomorphism, which is   proved in Theorem 
 ~\ref{K-1} geometrically by identifying $\mbf K(\mbf d^{\bullet})$ with the module $\mathcal K_{D^{\bullet}}$. 
Since  the map $x_1\theta^{\mbf d^1} x_2\mapsto x_1\theta^{\mbf d^1} x_2 \theta^{\mbf d^2}$ defines an isomorphism
 $M_0\otimes V_{\lambda^1} \to M_{\lambda^2}\otimes V_{\lambda^1}$, as vector spaces. So the morphism
 $\psi_{\lambda^{\bullet}} : M_{\lambda^2}\otimes V_{\lambda^1} \to \mbf K(\mbf d^{\bullet})$ is also an isomorphism.
 However,
$\psi'_{\lambda^{\bullet}}$ and  $\psi'_{\lambda^{\bullet}}$ are   Not  isomorphic  in  general. 
Indeed, we have $\mbf K(\mbf d^{\bullet}) =\mbf K(\mbf d)$ in $\mathfrak{sl}_2$ case.  Details will be given in a forthcoming paper.
\end{rem}

\section{Tensor product  varieties}

\subsection{$\E(V)$}
\label{E(v)}

Recall from Section ~\ref{definition} that $\Gamma=(I, H)$ is a graph.
For any finite dimensional  $I$-graded vector space $V$ over $\C$, we define
\[
\E(V) =\oplus_{h\in H} \Hom(V_{h'},V_{h'' } )
\]
to be the representation space of the graph $\Gamma$ of dimension vector $\nu=\sum_{i\in I}\dim V_i i$. 

An element $x\in \E(V)$ is called $nilpotent$ if there exists an integer $n$ such that for any sequence 
$h_1,\cdots, h_n$ of arrows such that $h_1''=h_2'$, $\cdots$, $h''_{n-1} =h_n'$, the composition 
$x_{h_n}x_{h_{n-1}} \cdots x_{h_1}$ is equal to $0$.

For any element $x \in \E(V)$ and  $I$-graded subspace  $W\subseteq V$, we say that $W$ is $x$-invariant if $x_h(W_{h'} )\subseteq W_{h''}$ for any $h\in H$.
In general, for any $I$-graded subspace $U$ in $V$,  we write
$\overline U^x$ for the smallest $x$-invariant $I$-graded subspace in $V$ containing $U$, and 
$\underline U^x$ for the largest $x$-invariant $I$-graded subspace contained in $U$. Whenever there is no ambiguity, we  write $\overline U$ and $\underline U$
for $\overline U^x$ and $\underline U^x$, respectively.

Suppose that $W$ is an $x$-invariant  $I$-graded subspace. Let $T=V/W$ and fix an isomorphism $V \tilde =W\oplus T$. 
Then  by restricting $x$ to $W$, we obtain an element 
$x^{WW}\in \E(W)$. By passage to quotient, we obtain an element  $x^{TT}\in \E(T)$. 
The element $x^{WT}$ is the collection of linear maps $x_h: T_{h'} \to W_{h''}$ for $h\in H$.
In other words, the $h$-component $x_h$ of $x$ can be written in the following block form:
\begin{equation}
\label{decom}
x_h: = 
\begin{pmatrix}
x^{WW}_h & x^{WT}_h\\
0 & x^{TT}_h.
\end{pmatrix}
: W_{h'}\oplus T_{h'}\to W_{h''}\oplus T_{h''}.
\end{equation}
Moreover, one can show that 
\begin{align}
\label{nilpotent}
\mbox{$x$ is nilpotent if and only if $x^{TT}$ and $x^{WW}$ are nilpotent.}
\end{align}

\subsection{$\E(V,D)$}
\label{EVD}
Recall that $\tilde \Gamma$ is the framed graph of $\Gamma$. 
We fix a finite dimensional $I^+$-graded space $D$.
Thus, the space 
\begin{equation}
\E(V, D) =\E(V) \oplus \Hom(D, V) \oplus \Hom (V, D),
\end{equation}
where $\Hom(V, D) =\oplus_{i\in I} \Hom(V_i, D_i)$ and $\Hom(D, V)$ is defined similarly,
is the representation space $\E(V\oplus D)$ of the framed graph $\tilde \Gamma$.
We shall identify $\E(V)$ with $\E(V, 0)$.

Elements in $\E(V, D) $ will be represented by $X=(x, p, q)$ with $x\in \E(V)$, $p\in \Hom(D, V)$ and $q\in \Hom(V, D)$.
We also write
\begin{equation}
\label{xi}
X(i)=\bigoplus_{\substack{h\in \tilde \Gamma \\ h'=i}} X_h=q_i \oplus \bigoplus_{\substack{h\in \Gamma\\h'=i}} x_h \quad \mbox{and}\quad
(i) X=\sum_{\substack{h\in \tilde \Gamma\\h''=i}} \epsilon(h) X_h=-p_i+\sum_{\substack {h\in \Gamma\\ h''=i}}\epsilon (h) x_h.
\end{equation}

The group $\G_V:=\prod_{i\in I} \mrm{GL}(V_i)$ acts on $\E(V, D)$ by conjugation, i.e.,
\[
g.X = X^g, \quad  X^g_h=
\begin{cases}
g_{h''} x_h g_{h'}^{-1}, & \mbox{if}\; h',h''\in I,\\
q_i g_i^{-1}, &  \mbox{if}\; h''\not \in I,\\
g_i p_i, & \mbox{if}\; h'\not \in I,
\end{cases}
\quad \forall g\in \G_V, X\in \E(V, D).
\]

\subsection{Lusztig's and Nakajima's quiver varieties}
\label{Lusztig}
Let $\epsilon: H\to \{\pm 1\}$ be a map satisfying $\epsilon(h)+\epsilon(\bar h) =0$ for any $h\in H$.
The variety $\mbf \Lambda_{V, D}$ is the subvariety of $\E(V, D)$ consisting of all elements $(x, p, q)$ such that 
\[
\sum_{h\in H: h''=i} \epsilon( h) x_h x_{\bar h} -p_i q_i =0, \quad \forall i\in I.
\]

Let $\mbf \Lambda_{V\oplus D} \subseteq \mbf \Lambda_{V, D}$ be Lusztig's Lagrangian nilpotent variety (\cite{Lusztig90b}, \cite{Lusztig91}) attached to the graph $\tilde \Gamma$.  More precisely,
this is a closed subvariety of
$\mbf \Lambda_{V, D}$ determined by the conditions that $X$ is nilpotent and  $q_i p_i=0$ for any $i\in I$. 
We write $\mbf \Lambda_{V}$ for $\mbf \Lambda_{V\oplus 0}$, which is a closed subvariety in $\E(V)$. Let
\[
\mbf L_{V, D}=\mbf \Lambda_V \times \Hom(V, D),
\]
and 
$\mbf L^s_{V, D}$ be the open subvariety consisting of all elements $X$ such that $X(i)$ (see (\ref{xi}))  are injective
for any $i\in I$.  It is well-known that $\mbf L_{V, D}$ and $\mbf L^s_{V,D}$ are pure dimensional, i.e. all irreducible components have the same dimension:
\begin{equation}
\label{dim-L}
\dim \mbf L_{V,D}=\dim \mbf L^s_{V, D}=\frac{1}{2}\dim \E(V,D).
\end{equation}

Note that $\G_V$ acts freely on $\mbf L^s_{V,D}$ and its quotient $\mathfrak L_{V, D}=\G_V\backslash \mbf L^s_{V, D}$ is the Lagrangian fiber of Nakajima's quiver variety (\cite{Nakajima94}).

\subsection{Tensor product varieties $\mbf \Pi_{V, D^{\bullet}}$}
\label{twocopytensor}
We  fix a decomposition 
\[
D=D^2\oplus D^1,\quad \mbox{with} \quad \dim D^2=d^2 \quad \mbox{and} \quad \dim D^1= d^1, 
\]
and a flag
$D^{\bullet}= (\check D^0\supseteq \check D^1 \supseteq \check D^2)$ where $D=\check D^0$,  $\check D^1=D^2$ and   $ \check D^2=0$.

Let $\mbf \Pi_{V, D^{\bullet}}$ be the closed subvariety of $\mbf \Lambda_{V, D}$ consisting of all nilpotent elements $X=(x, p, q)$ such that
 \begin{equation}
 \label{tensor-condition}
 \overline{p(D)} \subseteq \underline{q^{-1} (D^{2})}
 \quad \mbox{and}\quad 
 p(D^2)=0,
 \end{equation}
where the notations are defined in Section ~\ref{E(v)}.
(When $D^2=0$, the geometry of $\mbf \Pi_{V, D^{\bullet}}$ corresponds to the module $\mbf K'(\mbf d)$ in Section ~\ref{raw}.)

Let $\mbf \Pi_{V, D^{\bullet}; \nu^2}$ be the locally closed subvariety of $\mbf \Pi_{V, D^{\bullet}}$ consisting of all elements such that 
$\dim \underline{ q^{-1}(D^2)} =\nu^2$.

Fix an $I$-graded subspace $V^2$ of dimension $\nu^2$ in $V$.
Let $\mbf \Pi_{V, D^{\bullet}; V^2}$ be the closed subvariety of $\mbf \Pi_{V, D^{\bullet}; \nu^2}$ consisting of all elements such that 
$\underline{ q^{-1}(D^2)} =V^2$.

Let $V^1=V/V^2$ and fix a decomposition $V=V^2\oplus V^1$. Then $D^2\oplus V^2$ is $X$-invariant for any element $X\in \mbf \Pi_{V, D^{\bullet}; V^2}$.
Thus to each element $X=(x, p, q)\in \mbf \Pi_{V, D^{\bullet}; V^2}$, attached the elements
\begin{equation}
\label{X1X2}
X^1=(x^{V^1V^1}, p^{V^1 D^1}, q^{D^1 V^1})\in \mbf \Lambda_{V^1, D^1},\quad 
X^2=(x^{V^2V^2}, p^{V^2 D^2}, q^{D^2V^2})\in \mbf \Lambda_{V^2,D^2},
\end{equation}
and $(x^{V^2V^1}, p^{V^2 D^1}, q^{D^2V^1})$ just as in (\ref{decom}).
Since $X$ is nilpotent, so are $X^1$ and $X^2$ by (\ref{nilpotent}).
Moreover, by the condition (\ref{tensor-condition}), we see that 
\[
p^{V^1 D^1}=0,\quad p^{V^2 D^2}=0\quad \mbox{and} \quad \ker X^1(i)=0, \quad \forall i\in I.
\]

From the above analysis, the assignment $X\mapsto (X^1, X^2)$ defines a morphism
\begin{equation}
\label{rho}
\rho:\mbf  \Pi_{V, D^{\bullet}; V^2} \to\mbf  L^s_{V^1, D^1}\times \mbf L_{V^2,D^2}.
\end{equation}

\begin{lem}
\label{dimrho}
The morphism $\rho$ is a vector bundle of fiber  dimension 
\begin{equation}
\label{dim-rho}
\sum_{h\in H} \nu_{h'}^1\nu^2_{h''} + \sum_{i\in I} d^1_i \nu^2_i + \sum_{i\in I} \nu^1_i d^2_i -\sum_{i\in I} \nu^1_i\nu^2_i.
\end{equation}
\end{lem}

\begin{proof}
The proof is the same as the proof of Proposition 2.7 in ~\cite{M03}. By keeping in mind the result (\ref{nilpotent}), 
the fibre of $\rho$ under a fixed  point $(X^1, X^2)$ is isomorphic to the  vector subspace in 
$\oplus_{h\in H} \Hom(V^1_{h'}, V^2_{h''}) \oplus \Hom( D^1, V^2) \oplus \Hom(V^1, D^2)$
consisting of all elements
$(x^{V^2V^1}, p^{V^2D^1},q^{D^2V^1})$ such that 
\begin{equation}
\label{tensor-condition-B}
\sum_{h\in H: h''=i} \epsilon(h) (x_h^{V^2V^2} x_{\bar h}^{V^2V^1}+ x_h^{V^2V^1} x^{V^1V^1}_{\bar h})-
p_i^{V^2D^1} q^{D^1V^1}_i =0,\quad \forall i\in I.
\end{equation}
So the fibre $\rho^{-1}(X^1,X^2)$ is isomorphic to the kernel of  the following linear map
\[
f: \oplus_{h\in H} \Hom(V^1_{h'}, V^2_{h''}) \oplus \Hom( D^1, V^2) \oplus \Hom(V^1, D^2)\to \Hom(V^1, V^2)
\]
given by
\[
(x^{V^2 V^1}, p^{V^2D^1},q^{D^2V^1})\mapsto
\left (\sum_{h\in H: h''=i} \epsilon(h) (x_h^{V^2V^2} x_{\bar h}^{V^2V^1}+ x_h^{V^2V^1} x^{V^1V^1}_{\bar h})-
p_i^{V^2D^1} q^{D^1V^1}_i |i\in I \right ).
\]
The condition that $X^1$ is in $\mbf L^s_{V^1,D^1}$ implies that $f$ is surjective.  

Indeed,
let us consider the trace map 
\[
\tr:  \Hom(V^1, V^2) \times \Hom(V^2,V^1) \to \C,\quad
( p^{V^2V^1}, p^{V^1 V^2})\mapsto \sum_{i\in I} \tr( p^{V^2V^1}_ip^{V^1 V^2}_i),
\]
where  $\tr( p^{V^2V^1}_ip^{V^1 V^2}_i)$ is the trace of the endomorphism.
Let $f^{\perp}$ be the perpendicular space of $\im (f)$ with respect to the trace map $\tr$.
Given any $q^{V^1V^2}$ in $f^{\perp}$, we have 
\[
\sum_{h\in H: h''=i} \epsilon(h) (\tr ( x_h^{V^2V^2} x_{\bar h}^{V^2V^1}q_i^{V^1V^2} )+ \tr ( x_h^{V^2V^1} x^{V^1V^1}_{\bar h}q_i^{V^1V^2}))-
\tr( p_i^{V^2D^1} q^{D^1V^1}_iq_i^{V^1V^2})=0,
\]
for any triple $(x^{V^2V^1}, p^{V^2D^1},q^{D^2V^1})$ and $i\in I$. This implies that 
\[
\tr ( x_h^{V^2V^2} x_{\bar h}^{V^2V^1}q_i^{V^1V^2} )=0,\quad
 \tr ( x_h^{V^2V^1} x^{V^1V^1}_{\bar h}q_i^{V^1V^2})=0,\quad
 \tr( p_i^{V^2D^1} q^{D^1V^1}_iq_i^{V^1V^2})=0,
\]
for any triple $(x^{V^2V^1}, p^{V^2D^1},q^{D^2V^1})$, any $i\in I$ and $h\in H$ such that $h''=i$.
Since $x^{V^2V^1}$ and $p^{V^2D^1}$ are arbitrary, the last two equations imply that 
\[
x^{V^1V^1}_{\bar h}q_i^{V^1V^2}=0 \quad \mbox{and}\quad
 q^{D^1V^1}_iq_i^{V^1V^2}=0, \quad \forall i\in I, h\in H \; \mbox{such that }\; h''=i.
\]
This is to say that $\im(q^{V^1V^2}_i)\subseteq \ker X^1(i)$ for any $i\in I$. But the fact that  $X^1\in \mbf L^s_{V,D}$ implies that
$\im (q_i^{V^1V^2})=0$ for any $i\in I$. In other words, $f^{\perp}=\{0\}$. Since the trace map $\tr$ is non degenerate,  $f$ is surjective.

It is obvious that the dimension of the kernel of $f$ is exactly  (\ref{dim-rho}). The lemma follows.
\end{proof}

Given any variety $\mbf \Pi$, we denote by $\Irr \mbf \Pi$ the set of all irreducible components of $\mbf \Pi$.

\begin{prop} 
\label{property-rho}
The following statements hold.
\begin{enumerate}
\item $\mbf \Pi_{V, D^{\bullet}; V^2}$ has pure dimension of dimension  $\frac{1}{2} \dim \E(V, D) -\sum_{i\in I} \nu^1_i\nu^2_i$. Moreover, 
         $\# \Irr \mbf \Pi_{V, D^{\bullet}; V^2} =\# \Irr ( \mbf L^s_{V^1, D^1}\times \mbf L_{V^2, D^2})$.
\item $\mbf \Pi_{V, D^{\bullet};\nu^2}$ has pure dimension of dimension $\frac{1}{2} \dim \E(V, D)$.
\item $\mbf \Pi_{V, D^{\bullet}}$  has   pure dimension of dimension $\frac{1}{2} \dim \E(V, D)$.
\item $\# \Irr \mbf \Pi_{V, D^{\bullet}} =\sum_{V^1, V^2} \# \Irr ( \mbf   L^s_{V^1,D^1} \times \mbf  L_{V^2,D^2})$, 
         where the sum runs over the representatives $(V^1, V^2)$ of $(\nu^1, \nu^2)$ such that $\nu^1+\nu^2=\nu$.
\end{enumerate}
\end{prop}

\begin{proof}
The first statement follows from the fact that $\rho$ is a vector bundle and that $\mbf L_{V, D}$ and $\mbf L^s_{V, D}$ are  of pure dimension. 
The dimension of $\mbf \Pi_{V, D^{\bullet}; V^2}$ can be computed by using (\ref{dim-L}) and (\ref{dim-rho}).

Let $\mrm{Gr}(\nu^2, V)=\times_{i\in I} \mrm{Gr}(\nu^2_i, V_i)$ be the product of the Grassmannian of $\nu^2_i$-dimensional vector subspaces in $V_i$ for $i\in I$.
The assignment $X\mapsto \underline {q^{-1}(D^2)}$ defines a surjective morphism 
$\mbf \Pi_{V, D^{\bullet}; \nu^2}\to \mrm{Gr}(\nu^2,V)$. It is clear that the fibre  of the morphism at $V^2$ is $\mbf \Pi_{V, D^{\bullet}; V^2}$ and the dimension
of $\mrm{Gr}(\nu^2,V)$ is $\sum_{i\in I} \nu^1_i\nu^2_i$. 
The second statement then follows. 

Since $\mbf \Pi_{V, D^{\bullet}}= \sqcup_{\nu^2} \mbf \Pi_{V, D^{\bullet}; \nu^2}$, the third statement follows from the second one. So is the forth statement.
The proposition is proved.
\end{proof}

\subsection{Finer structure on $\mbf \Pi_{V, D^{\bullet}}$}

For any $i\in I$ and $n\in \mbb N$, let 
\[
_{i, n}\mbf \Pi_{V, D^{\bullet}}=
\{ X\in \mbf \Pi_{V, D^{\bullet}}| \dim V_i/ (i)X=n\},
\]
where $(i)X$ is defined in (\ref{xi}). Suppose that 
the vector subspace $W$ in $V$ has dimension $\nu-ni$.
Denote by $\IHom (V, W)$ the subset of $\Hom(V, W)$ consisting of injective maps.
Let 
$\mbf \Pi$ be the subvariety of $_{i,n}\mbf \Pi_{V, D^{\bullet}} \times \, _{i,0} \mbf \Pi_{W, D^{\bullet}}\times \IHom(V, W)$ consisting of all triples 
$(X, X^1, R)$ such that
\[
R_{h''} X^1_h=X_hR_{h'}, \forall h \in \tilde \Gamma,
\]
where $R_{h'}$ and $R_{h''}$ are understood to be identity morphism if $h'$ or $h''$ are in $I^+$.
Then we have a diagram
\begin{equation}
\label{finer-action}
\begin{CD}
_{i, n}\mbf \Pi_{V, D^{\bullet}} @<\pi_1<< \mbf \Pi @>\pi_2 >> _{i, 0}\mbf \Pi_{W, D^{\bullet}},
\end{CD}
\end{equation}
where $\pi_1$ and $\pi_2$ are the obvious projections.
Moreover, $\pi_2$ factors through $ _{i,0} \mbf \Pi_{W, D^{\bullet}}\times \IHom(V, W)$ via the obvious projections
$\pi_2': (X, X^1, R)\mapsto (X^1, R)$ and $\pi_2'': (X^1, R)\to X^1$.

\begin{lem} 
\label{finer-lem}
We have 
\begin{enumerate}
\item $\pi_1$ is a principal  $\G_W$-bundle.

\item $\pi_2'$ is a vector bundle of dimension $\sum_{h\in H: h'=i} n \nu_{h''} + n d_i -n (\nu_i -n)$.

\item $\pi_2''$ is a trivial bundle of dimension $\dim \G_W+ n(\nu_i-n)$.

\end{enumerate}
\end{lem}

\begin{proof}
The first and third statements are clear. The second one follows from  an argument similar to the proof of Lemma ~\ref{dimrho} and ~\cite[12.5]{Lusztig91}.
\end{proof}

The following proposition follows from Lemma ~\ref{finer-lem}.

\begin{prop} 
\label{pi-induction}
Suppose that $\dim W + ni=\dim V$. 
The assignment $Y\mapsto \pi_2(\pi_1^{-1} (Y))$ defines a bijection
\begin{align}
\te_i^{\max} : \Irr _{i, n} \mbf \Pi_{V, D^{\bullet}} \to \Irr _{i, 0}\mbf \Pi_{W, D^{\bullet}}.
\end{align}
The assignment $ Y' \mapsto \pi_1(\pi_2^{-1}(Y')) $ defines a bijection 
\begin{align}
\label{fn}
\tf_i^n:   \Irr _{i,0}\mbf \Pi_{W, D^{\bullet}} \to \Irr _{i, n}\mbf \Pi_{V, D^{\bullet}}.
\end{align}
Moreover the maps $\te_i^{\max}$ and $\tf_i^{n}$  are inverse to each other.
\end{prop}

\subsection{Crystal structure on  the $\mbf \Pi_{V,D^{\bullet}}$'s}
\label{IrrD}

For each $\nu\in \mbb N[I]$, we fix  an $I$-graded vector space $V(\nu)$ of dimension $\nu$.
Let 
\[
\Irr (D^{\bullet})=\sqcup_{\nu\in \mbb N[I]} \Irr \mbf \Pi_{V(\nu), D^{\bullet}}.
\]
For any $Y\in \mbf \Pi_{V, D^{\bullet}}$ and $i\in I$, we define 
\begin{align*}
& \wt(Y) = \lambda -\nu \in \X,\\
& \e_i(Y) =n,\quad \mbox{if $Y\cap\, _{i, n}\mbf \Pi_{V, D^{\bullet}}$ is open dense in $Y$},\\
& \varphi_i=\e_i(Y)+ (i, \wt(Y)),
\end{align*}
where $\lambda \in \X$ is a fixed element such that $\lambda(i)=\dim D_i$ for all $i\in I$, $\nu\in \X$ is via the imbedding 
$\mbb N[I]\to \X$.
We also define
\begin{align*}
\tf_i(Y) =\tf_i^{\e_i(Y)+1} \te^{\max}_i(Y) 
\quad\mbox{and}\quad
\te_i(Y) =
\begin{cases}
\tf_i^{\e_i(Y)-1} \te_i^{\max} (Y) & \mbox{if}\; \e_i(Y)>0,\\
 0 &\mbox{otherwise}.
 \end{cases}
\end{align*}
In this way, we have defined the following five maps on $\Irr (D^{\bullet})$:
\[
\wt: \Irr(D^{\bullet}) \to \X, \;
\e_i, \varphi_i: \Irr (D^{\bullet}) \to \mbb Z
\; \mbox{and}\;
\te_i, \tf_i: \Irr (D^{\bullet})\to \Irr(D^{\bullet})\sqcup \{0\},\quad \forall i\in I.
\]

The following lemma follows from a straightforward checking.

\begin{lem}
\label{crystal-pi}
The data  $(\Irr (D^{\bullet}), \mathrm{wt}, \e_i, \varphi_i, \te_i, \tf_i)_{i\in I}$ form a crystal defined in Section ~\ref{crystal}.
Moreover, the crystal $\Irr (D^{\bullet})$ is generated by the set of all elements $Y$ such that $\e_i(Y)=0$ for all $i\in I$.
\end{lem}

We set 
\[
\Irr (\mbf L^s_D) =\sqcup_{\nu\in \mbb N[I]} \Irr \mbf L^s_{V(\nu), D}
\quad \mbox{and} \quad 
 \Irr (\mbf L_D) =\sqcup_{\nu\in \mbb N[I]} \Irr \mbf L_{V(\nu), D}.
\]
On $\Irr(\mbf L^s_D)$ and $\Irr (\mbf L_D)$, one may define crystal structures in exactly the same way as the crystal structure on $\Irr (D^{\bullet})$.
Moreover, 
the crystal structure on $\Irr (\mbf L^s_D)$ is isomorphic to $B(\lambda)$, while $\Irr(\mbf L_D)$ isomorphic to $B(\lambda, \infty)$.
See ~\cite{KS97} and ~\cite{Nakajima94} for details.

Define a  map 
\begin{equation}
\label{psi}
\psi: \Irr (\mbf L^s_{D^1}) \otimes \Irr (\mbf L_{D^2}) \to \Irr (D^{\bullet})
\end{equation}
as follows.
For any $Y^1\in \Irr \mbf L^s_{V^1, D^1}$ and $Y^2\in \Irr \mbf  L_{V^2, D^2}$, we define $\psi(Y^1\otimes Y^2)$ to be the irreducible component 
$Y\in \Irr (\mbf \Pi_{V, D^{\bullet}})$ such that $Y\cap \mbf \Pi_{V, D^{\bullet}; V^2}=\rho^{-1}(Y^1\times  Y^2)$ where $\rho$ is defined in (\ref{rho}).
By Proposition ~\ref{property-rho} (4), the map $\psi$ is a bijection. Moreover,

\begin{thm}
\label{Pi-crystal}
 $\Irr (D^{\bullet})$ is isomorphic to $\Irr (\mbf L^s_{D^1})\otimes \Irr (\mbf L_{D^2})=B(\lambda^1)\otimes B(\lambda^2, \infty)$ as crystals via $\psi$, where $\lambda^1$ and $\lambda^2$ are fixed elements such that $\lambda^a(i) =\dim D_i^a$ for any $i\in I$ and $a=1, 2$.
\end{thm}

\begin{proof}
The proof is similar to the proof of Theorem 4.6 in ~\cite{Nakajima01}. 
We have to show that $\psi$ is compatible with the five maps on $\Irr (D^{\bullet})$ and $\Irr (\mbf L^s_{D^1})\otimes \Irr (\mbf L_{D^2})$, respectively.
It is obvious that $\wt(\psi(Y^1\otimes Y^2)) =\wt(Y^1) + \wt(Y^2)$. 

Choose a triple $(X, X^1, X^2)$ such that $\rho(X) =(X^1,X^2)$ in (\ref{rho}). Let us fix a decomposition $V=V^2\oplus V^1$. 
Then the canonical short exact sequence $0\to V^2\to V \to V^1 \to 0$ induces a complex
\begin{equation}
\label{tensor-thm-1}
\ker (i) X^2 / \im X^2(i) \to \ker (i) X /\im X(i) \overset{\tilde b}{\to} \ker (i) X^1/\im X^1(i).
\end{equation}
The condition (\ref{tensor-condition}) gives rise to a linear map 
\[
c: \ker (i) X^1 /\im X^1(i) \to V_i^2/\im (i) X^2,
\]
and $\tilde bc=0$. So $\tilde b$ factors through $\ker (c)$. Thus the complex (\ref{tensor-thm-1}) induces a complex
\begin{equation}
\label{tensor-thm-2}
0\to \ker (i) X^2/\im X^2(i) \to \ker (i) X/\im X(i) \to \ker (c)\to 0,
\end{equation}
which is indeed a short exact sequence.
From (\ref{tensor-thm-2}), one deduces the following lemma.

\begin{lem}
\label{tensor-thm-A}
Assume that  $\psi(Y) =Y^1\otimes Y^2$ satisfies $\e_i(Y) =0$, then 
$\e_i(Y^1) =0$ and $\dim \ker (c) = \wt_i (Y^1) -\e_i(Y^2)$.
\end{lem}

Suppose again that $\e_i(Y)=0$ and that  $X\in Y$ is a generic point.
Fix a vector space $T$ of dimension $ri$ and $\tilde V= V\oplus T$. We consider  the variety 
\[
\mbf \Pi_X= \{ \tilde X\in \mbf \Pi_{\tilde V, D^{\bullet}}: \tilde X|_ V=X\}.
\]
Just like (\ref{tensor-condition-B}), we deduce that $\mbf \Pi_X$ can be identified with the vector subspace in 
\[
\Hom( T, D_i) \oplus_{h\in H: h''=i} \Hom (T_i, V_{h'})
\] 
consisting of all vectors $(q^{DT}_i, x_{\bar h}^{VT})_{h\in H: h''=i}$ such that 
\begin{equation}
\label{tensor-thm-3}
\sum_{h\in H: h''=i} \epsilon_i(h) x_h x^{VT}_{\bar h} - p_i q_i^{DT}=0
\end{equation}
From the equation (\ref{tensor-thm-3}), one can deduce that the projection 
$q_i^{D^1T}+ \sum_{h\in H: h''=i} x^{V^1T}_{\bar h}: T\to D_i^1\oplus \oplus_{h\in H: h''=i} V_{h'}^1$ of $q^{DT}$ and $x_{\bar h}^{VT}$ to $D_i^1$ and 
$V_{h'}^1$, respectively,  has image contained in $\ker (i) X^1$. Moreover, 
the composition of this map with the projection from $\ker (i) X^1$ to $\ker (i) X^1/\im X^1(i)$ and the map $c$ vanishes. So  it gives rise to a linear map $ \phi:T\to \ker (c)$.
The assignment $(q^{DT}, x_{\bar h}^{VT}) \mapsto \phi$ then defines  a linear map
\[
f: \mbf \Pi_{X} \to \Hom (T, \ker (c)).
\]
Moreover, the map $f$ is surjective. In fact, for any $\phi\in \Hom (T, \ker (c))$, one may define an element in $\mbf \Pi_X$ as follows.
Fix a decomposition $\ker (i) X^1 = \im X^1(i) \oplus \ker (i) X^1/\im X^1(i)$. 
Let $p_i^{D^1T}$ be the composition of $\phi$ and $\ker (c) \to \ker (i) X^1/\im X^1(i) \to \ker (i) X^1 \to D_i^1$. Similar 
$x_{\bar h}^{V^1T}$ be the composition of $\phi$ and $\ker (c) \to \ker (i) X^1/\im X^1(i) \to \ker (i) X^1 \to V_{h'}^1$.
Then the condition (\ref{tensor-thm-3}) determines uniquely an element  $x_{\bar h}^{V^2T}$. Now choose an arbitrary element $p_i^{D^2T}$.
The element $(q^{DT}_i, x_{\bar h}^{VT})$ in $\mbf \Pi_{X}$ is thus obtained, moreover, the element gets sent to $\phi$ via $f$. Therefore $f$ is surjective.

Let $ \Hom (T, \ker (c))^1$ be the open dense subvariety of $ \Hom (T, \ker (c))$ consisting of all elements of maximal rank.
Since $f$ is surjective, we have that the inverse image $\mbf \Pi_X^1$ of $ \Hom (T, \ker (c))^1$ under $f$ is open dense in $\mbf \Pi_X$.
Let $\mbf \Pi_X^0$ be the subvariety in $\mbf \Pi_X^1$ such that the corresponding element $\phi \in  \Hom (T, \ker (c))$ satisfying $\ker \phi =\ker (q_i^{D^1T}+ \sum_{h\in H:h''=i} x_{\bar h}^{V^1T})$.
Then from the above explicit construction of elements in $\mbf \Pi_X$, for any given $\phi$, we see that 
\begin{equation}
\label{tensor-thm-4}
\mbox{$\mbf \Pi_X^0$ is open dense in $\mbf \Pi_X$ and, moreover, $(\underline {\tilde q)^{-1} (D^2)} = V^2\oplus \ker \phi$ for any $ \tilde X \in \mbf \Pi_X^0$.}
\end{equation}
Here $\tilde X=(\tilde x, \tilde p, \tilde q)$ and  $V^2=\underline { q^{-1}(D^2) }$ for the fixed element  $X$.
From (\ref{tensor-thm-4}) and using (\ref{finer-action}),  one can show the following lemma.

\begin{lem}
\label{tensor-thm-B}
Suppose that $Y$ is an irreducible component in $\Irr (D^{\bullet})$ such that $\psi(Y)=Y^1\otimes Y^2$ and $\e_i(Y)=0$. Then one has 
\begin{equation}
\tf_i^r (Y) 
=\begin{cases}
\tf_i^r Y^1\otimes Y^2  & \mrm{if}\; r  \leq \wt_i (Y^1 )-\e_i(Y^2),\\
\tf_i^{\wt_i(Y^1) -\e_i(Y^2) }Y^1 \otimes \tf_i^{r- \wt_i(Y^1) + \e_i(Y^2) } Y^2 & \mrm{if}\; r> \wt_i(Y^1)-\e_i(Y^2).
\end{cases}
\end{equation}
\end{lem}

Now by using Lemmas ~\ref{crystal-pi}, ~\ref{tensor-thm-A} and ~\ref{tensor-thm-B} and using the proof of Lemma 4.11 in ~\cite{Nakajima01}, one can show that $\psi$ is compatible with the five maps.
\end{proof}

\subsection{Crystal structure on the $\mbf \Pi_{V, D^{\bullet}}^s$'s}

Let 
\[
\mbf \Pi_{V, D^{\bullet}}^s=\{ X\in \mbf \Pi_{V, D^{\bullet}} | X(i) \, \mbox{are injective for all $i\in I$}\}.
\]
This is an open variety in $\mbf \Pi_{V, D^{\bullet}}$. 
Similarly, one can define the open subvariety $\mbf \Pi_{V, D^{\bullet}; V^2}^s$ in $\mbf \Pi_{V, D^{\bullet}; V^2}$.
Then we have the following cartesian diagram
\begin{equation}
\begin{CD}
\mbf \Pi_{V, D^{\bullet}; V^2}^s @>\rho^s>> \mbf L_{V^1, D^1}^s \times \mbf L_{V^2, D^2}^s\\
@VVV @VVV\\
\mbf \Pi_{V, D^{\bullet}; V^2} @>\rho>> \mbf L_{V^1, D^1}^s \times \mbf L_{V^2, D^2},
\end{CD}
\end{equation}
where the bottom row is (\ref{rho}).
From this diagram, we have 

\begin{prop}
\label{stable}
Statements similar to (1), (2) and (3) in Proposition ~\ref{property-rho} hold for the varieties
$\mbf \Pi_{V, D^{\bullet}; V^2}^s$, $\mbf \Pi_{V, D^{\bullet}; \nu^2}^s$ and $\mbf \Pi_{V, D^{\bullet}}^s$. Moreover,
\[
\# \Irr \mbf \Pi_{V, D^{\bullet}}^s =\sum_{\nu^1, \nu^2}\# \Irr  (\mbf L_{V^1(\nu^1), D^1}^s \times \mbf L_{V^2(\nu^2), D^2}^s),
\]
where the sum runs over all pairs $(\nu^1,\nu^2)$ such that $\nu^1+\nu^2=\nu$.
\end{prop}

 Note that  the $\G_V$-action on $\mbf \Pi^s_{V, D^{\bullet}}$ (resp. $\mbf L^s_{V, D}$) is free and its quotient $\tilde{\mathfrak Z}_{V, D^{\bullet}}$ (resp. $\mathfrak L_{V, D}$) is the tensor product variety defined by 
 Nakajima ~\cite{Nakajima01} and Malkin ~\cite{M03} (resp. Nakajima ~\cite{Nakajima94}).

Let 
\[
\Irr (D^{\bullet})^s =\sqcup_{\nu\in \mbb N[I]} \Irr (\mbf \Pi_{V(\nu), D^{\bullet}}^s)
\]
The inclusions $\mbf \Pi_{V, D^{\bullet}}^s\subseteq \mbf \Pi_{V, D^{\bullet}}$ define a surjective map, via restriction,
\[
\iota: \Irr (D^{\bullet}) \to \Irr (D^{\bullet})^s. 
\]
Similar to the crystal structure on $\Irr (D^{\bullet})$, one can define a crystal structure on $\Irr (D^{\bullet})^s$ with the help from a diagram similar to (\ref{finer-action}). 
It is proved in Theorem 4.6 in ~\cite{Nakajima01} that the crystal structure on $\Irr (D^{\bullet})^s$ is isomorphic to the 
crystal on the tensor product $B(\lambda^1)\otimes B(\lambda^2)$.

\begin{cor}
The surjective map $\iota: \Irr (D^{\bullet}) \to \Irr (D^{\bullet})^s$ is the crystal morphism $B(\lambda^1)\otimes B(\lambda^2, \infty) \to B(\lambda^1) \otimes B(\lambda^2)$.
\end{cor}

Note that the crystal structure on $\Irr (D^{\bullet})^s$ has been studied in ~\cite{M03}, ~\cite{Nakajima01} and ~\cite{S02}.

\section{Induction and restriction}
We shall recall Lusztig's induction and restriction functors in this section. We refer to ~\cite{Lusztig90}-\cite{Lusztig93} for further reading.

\subsection{Lusztig's diagrams}
\label{diagram}
We preserve the setting in Section ~\ref{EVD}. 
Let $\Omega$ be an orientation of the framed graph $\tilde \Gamma$. 
This is a subset of the edge set $\tilde H$ of $\tilde \Gamma$ such that $\Omega \cup \overline{\Omega} =\tilde H$ and $\Omega \cap \overline{\Omega}=\mbox{\O}$.

Let $\E_{\Omega} (V, D)$ be the subspace of $\E(V, D)$ consisting of all elements $X$ such that $X_h=0$ for any $h\notin \Omega$.

Recall that we fix a decomposition $D=D^2\oplus D^1$ and $V=W\oplus T$.  Let
\[
\mbf F =\{ X\in \E_{\Omega}(V, D) | D^2\oplus W \; \mbox{is $X$-invariant}\}.
\]
Then we have the following diagram
\begin{equation}
\label{restriction-diag}
\begin{CD}
\E_{\Omega}(T, D^1) \times \E_{\Omega}(W, D^2)  @<\kappa << \mbf F @>\iota >> \E_{\Omega}(V, D),
\end{CD}
\end{equation}
where 
$\iota$ is the closed embedding and $\kappa (X) =(X^1, X^2)$ with $X^1$ and $X^2$ defined in a similar way as  (\ref{X1X2}).
Note that $\kappa$ is a vector bundle of fiber dimension 
$\sum_{h\in \Omega} \dim \tilde T_{h'} \dim \tilde  W_{h''}$ where $\tilde T_{h'}=T_{h'}$ if $h'\in I$ and $\tilde T_{h'} = D^1_{h'}$ if $h'\in I^+$.

Let $\mbf P_V$ be the stabilizer of $W$ in $\G_V$ and $\mbf R_V$ its unipotent radical.  Then $\mbf P_V/ \mbf R_V\simeq \G_T\times \G_W$. Consider 
the following diagram
\begin{equation}
\label{induction-diag}
\begin{CD}
\E_{\Omega}(T, D^1) \times \E_{\Omega}(W, D^2) @<\pi_1 << \G_V \times_{\mbf R_V}  \mbf F @>\pi_2 >> \G_V \times_{\mbf P_V} \mbf F @>\pi_3 >> \E_{\Omega}(V, D),
\end{CD}
\end{equation}
where 
$\pi_1(g, X) = \kappa (X) $, $\pi_2(g, X) = (g, X)$ and $\pi_3(g, X) =g.X$.
It is well-known that $\pi_3$ is proper, $\pi_2$ is a $\G_T\times \G_W$-principal bundle and $\pi_1$ is a smooth morphism of connected fiber with fiber dimension
$f_1= \dim \G_V/\mbf R_V+ \dim \kappa^{-1}(X^1, X^2)$.

\subsection{Notations in derived categories}
We recall some notations on derived categories from ~\cite{BBD82} and ~\cite{Lusztig93}.
Given any algebraic variety $X$ over $\mbb C$, 
denote by $\mathcal{D}(X)$ the bounded derived category of complexes 
of  constructible sheaves on $X$.
Denote by  $\mbb C_X$ the constant sheaf on $X$, regarded as  
the complex concentrated on degree zero.
Let $[-]$ be the shift functor. 
Let $f: X\to Y$ be a morphism of varieties, denote by
$f^*: \mathcal D(Y) \to \mathcal D(X)$ and 
$f_!: \mathcal D(X) \to \mathcal D(Y)$ 
the inverse image functor and the direct image functor with compact support, respectively.
Let $G$ be a connected algebraic group. Assume that $G$ acts on $X$
algebraically. Denote by $\mathcal D_G(X)$ the full subcategory of
$\mathcal D(X)$ consisting of all $G$-equivariant complexes over
$X$.
Similarly, denote by $\mathcal M_G(X)$ the  category of
 of all $G$-equivariant perverse sheaves on $X$.
If $G$ acts on $X$ algebraically and $f$ is a principal $G$-bundle,
then $f^*$ induces a functor, still denote by $f^*$, of equivalence
between $\mathcal M(Y)[\dim G]$ and $\mathcal M_G(X)$.
Its inverse functor is denoted by $f_{\flat}: \mathcal M_G(X) \to
\mathcal M(Y)[\dim G]$.

\subsection{Induction and restriction functors}
From the diagram (\ref{restriction-diag}), we form the functor
\begin{equation}
\label{restriction-functor}
\widetilde{\Res}^V_{T, W} = \kappa_! \iota^*: \D (\E_{\Omega}(V, D) ) \to \D(\E_{\Omega}(T, D^1)\times \E_{\Omega}(W, D^2)) .
\end{equation}
From the diagram (\ref{induction-diag}), we form the functor
\begin{equation}
\label{induction-functor}
\widetilde{\Ind}^V_{T,W}=\pi_{3!} \pi_{2\flat} \pi_1^* :  
\D_{\G_T\times \G_W} (\E_{\Omega}(T, D^1)\times \E_{\Omega}(W, D^2))  \to \D (\E_{\Omega}(V, D) ).
\end{equation}

The following shifted versions of the restriction and induction functors will also be used later.
\begin{equation}
\label{shifted}
\Res^V_{T,W} =\widetilde{\Res}^V_{T, W} [f_1-f_2-2 \dim \G_V/\mbf P_V],\quad
\Ind^V_{T, W} = \widetilde{\Ind}^V_{T,W}[f_1-f_2],
\end{equation}
where $f_1$ and $f_2$ are the fiber dimensions of the morphisms $\pi_1$ and $\pi_2$, respectively.
For simplicity, we write
\[
K_1\cdot K_2=\Ind^V_{T, W}(K_1\boxtimes K_2).
\]

\subsection{Special cases} 
\label{special}

The following special cases of the functors $\Res^V_{T, W}$ and $\Ind^V_{T,W}$ will be used extensively later.

The first case is when $D^1=0$ and $\dim T= ni$. In this case, $\E_{\Omega}(T, D^1)=0$. The shifted  induction and restriction functors induce the following functors
\begin{equation}
\label{iR}
_i  \R^{(n)}:  \D (\E_{\Omega}(V, D) ) \to \D( \E_{\Omega}(W, D)),  \quad 
\F_i^{(n)}: \D( \E_{\Omega}(W, D))  \to \D (\E_{\Omega}(V, D) ).
\end{equation}

The second case is when $D^2=0$ and $\dim W=ni$.  In this case, $\E_{\Omega}(W, D^2)=0$. The shifted restriction functor induces the following functor
\begin{equation}
\R^{(n)}_i:  \D (\E_{\Omega}(V, D) ) \to \D( \E_{\Omega}(T, D)).
\end{equation}

The third case is when $T=0$. In this  case,  $\E_{\Omega}(T, D^1)=0$ and  $\pi_3=\iota$ is a closed embedding. 
Thus, the functor $\Ind^V_{T, W}$ induces a fully faithful functor 
\begin{equation}
\label{L_d.}
L_{d^1} \cdot:  \D( \E_{\Omega}(V, D^2))  \to \D (\E_{\Omega}(V, D) ), \quad K_2\mapsto L_{d^1} \cdot K_2.
\end{equation}

The last case is when $W=0$. In this case, $\E_{\Omega}(W, D^2)=0$ and $\pi_3$ is the identity morphism. Thus the functor $\Ind^V_{T, W}$ induces a fully faithful functor
\begin{equation}
\label{L_d}
\cdot L_{d^2}:  \D( \E_{\Omega}(T, D^1))  \to \D (\E_{\Omega}(V, D) ),\quad K_1\mapsto K_1\cdot L_{d^2}.
\end{equation}

\section{Geometric study of $M_{\lambda}$}

\subsection{Perverse sheaves on $\E_{\Omega}(V)$}

Let 
\[
\E_{\Omega}(V) = \E_{\Omega}(V, 0),
\]
where $\E_{\Omega}(V, 0)$ is $\E_{\Omega}(V, D)$ in Section ~\ref{diagram} for $D=0$.

Let $\mbf i =(i_1,\cdots, i_m)$ be a sequence of vertices in $\Gamma$ and $\mathbf a=(a_1,\cdots, a_m)$ a sequence of non negative integers such that
$a_1 i_1+\cdots+ a_m i_m=\nu$.  
A flag $F^{\bullet}=(V =F^0\supseteq F^1\supseteq \cdots \supseteq F^m=0$ is of type $(\mbf{i, a})$ 
if $\dim F^l/F^{l+1}=a_{l+1}i_{l+1}$ for all $0\leq l\leq m-1$.
The variety $\tF_{\mbf{i, a}}$ is the variety of all pairs $(x, F^{\bullet}) $, 
where $x\in \E_{\Omega}(V)$ and $F^{\bullet}$ is a flag of type $(\mbf{i, a})$, 
such that each subspace in $F^{\bullet}$ is $x$-invariant. Let 
\begin{equation}
\label{proper}
\pi_{\mbf{i, a}} : \tF_{\mbf{i, a}} \to \E_{\Omega}(V)
\end{equation}
be the projection to the first component. We set
\begin{align}
\label{Lia}
L_{(\mbf{i, a})}=\pi_{\mbf{i, a}!} ( \mbb C_{  \tF_{\mbf{i, a}} }[\dim  \tF_{\mbf{i, a}} ]).
\end{align}

Since $\tF_{\mbf {i, a}}$ is smooth and $\pi_{\mbf{i, a}}$ is proper,  
the complex $L_{(\mbf {i, a})}$ is semisimple by the decomposition theorem in ~\cite{BBD82}.
Let $\mathcal Q_V$ be the full `semisimple'  subcategory of  $\D(\E_{\Omega}(V))$ whose simple objects are isomorphic to those appeared in $L_{(\mbf{i,a })}$ for various pairs $(\mbf{i, a})$.

Let $\N_{V, i}$ be the full subcategory of $\Q_V$ whose simple objects are isomorphic to those appeared in $L_{(\mbf{i,a })}$ for various pairs $(\mbf{i, a})$ such that
the last term $i_m$ of $\mbf i$ is $i$ and $a_m \geq d_i +1$.

Let
\[
 \E_{\Omega, i, \geq d_i+1}(V) :=\{ x\in \E_{\Omega}(V) | \dim \ker x(i)  \geq d_i +1 \},
\]
where $x(i)$ is defined in (\ref{xi}).

\begin{lem}  
Suppose that $i$ is a source in $\Omega$. Then 
any complex $K$ in $\mathcal Q_V$ is in $\mathcal N_{V, i}$ if and only if 
$\Supp(K) \subseteq  \E_{\Omega, i, \geq d_i+1}(V)$.
\end{lem}
This can be shown by the detailed analysis in ~\cite[9.3]{Lusztig93}.

Let $\N_V$ be the full subcategory of $\Q_V$ generated by 
$\mathcal N_{V, i}$ for all $i\in I$. Thus any simple object $K\in \Q_V$ is in $\N_V$ if and only if 
\begin{equation}
\label{a}
\Supp (\Phi_{\Omega}^{\Omega_i}K) \subseteq   \E_{\Omega_i, i, \geq d_i+1}(V), \quad \mbox{for some $i$ in $I$}.
\end{equation}
where $\Omega_i$ is an orientation of $\tilde \Gamma$ with $i$ a source and 
\[
\Phi_{\Omega}^{\Omega_i}: \D(\E_{\Omega}(V)) \to \D(\E_{\Omega_i}(V)),
\] 
is a fixed Fourier transform from $ \D(\E_{\Omega}(V))$ to $ \D(\E_{\Omega_i}(V))$ (see ~\cite{Lusztig93}, ~\cite{Lusztig91}, ~\cite{KS90}).

\subsection{Perverse sheaves on $\E_{\Omega}(V, D)$}
\label{framing}

Fix an order, say $(i^+_1, i^+_2,\cdots , i^+_N)$, of the set $I^+$.  Let $d$ denote
a fixed pair of sequences $(i^+_1, \cdots,i^+_N)$ and $(d_1, \cdots, d_N)$.
Let $(\mbf i, \mbf a)\cdot d$  be the composition  of the pair $(\mbf{i, a})$ and the pair $d$.

Then the complex $L_{(\mbf{i, a})\cdot d}$ is well defined and semisimple over $\E_{\Omega}(V, D)$.
Moreover, the complex $L_{(\mbf{i, a})\cdot d}$ is independent of the choice of the order $(i_1^+, \cdots, i_N^+)$ of $I^+$.
This is because $L_{(\mbf{i, a})\cdot d}$ is  equal to $L_{\mbf{i, a}} \cdot L_d$ where $\cdot L_d$ is defined in Section ~\ref{special}.
 Observe that $\E_{\Omega}(0, d)$ is a single point.
 $L_d=\mbb C_{\E_{\Omega}(0, d)}$, which is independent of the choice of the order on $I^+$. 

Let $\Q_{V, D}$ be the full subcategory of $\D(\E_{\Omega}(V, D))$ defined with respect to the complexes $L_{(\mbf{i, a})\cdot d}$ for various $(\mbf{i, a})$
in a similar way as $\Q_V$  to the complexes $L_{(\mbf{i, a})}$.

Let $\pi: \E_{\Omega}(V, D) \to \E_{\Omega}(V)$ be the obvious projection.  It then induces a functor, the shifted inverse image functor, 
\[
\pi^*[\nu d_{\Omega}] : \D(\E_{\Omega}(V)) \to \D(\E_{\Omega}(V, D)),  \quad \mbox{where}\; \nu d_{\Omega}= \sum_{i\in I: i\to i^+ \in \Omega} \nu_i d_i.
\]

\begin{lem}
\label{L=pi}
We have $\cdot L_d=\pi^*[\nu d_{\Omega}]$; moreover, they are functors of equivalence from $\Q_V$ to $\Q_{V, D}$.
\end{lem}

In fact, the vector subspace $0\oplus D$ is the only vector subspace in $V\oplus D$ of dimension $\sum_{i\in I} d_i i^+$ that is 
invariant under a fixed element $X$ in $\E_{\Omega}(V, D)$. 
The isomorphism $\cdot L_d =\pi^* [\nu d_{\Omega}]$  follows from this and the definition of 
the  multiplication ``$\mrm{Ind}^V_{T, W}$'' in Section ~\ref{diagram}.
They are equivalent because $\pi^*$ is a fully faithful functor due to the fact that  
$\pi$ is a trivial vector bundle of fiber dimension $\nu d_{\Omega}$.

Let $\N_{V, D}= \pi^* [\nu d_{\Omega}](\N_V)$, i.e., the full subcategory
of $\Q_{V, D}$ whose objects are of the form $\pi^* (K) $ with $K\in \N_V$.
By Lemma ~\ref{L=pi},  the condition (\ref{a}) can be restated as follows. 

\begin{lem}
 Any simple  object $K\in \Q_{V, D}$ is in $\N_{V, D}$ if and only if 
\begin{equation}
\label{b}
\Supp (\Phi_{\Omega}^{\Omega_i}K) \subseteq   \E_{\Omega_i, i, \geq 1}(V, D), \quad \mbox{for some $i$ in $I$},
\end{equation}
where  $\E_{\Omega, i, \geq 1}(V, D)=\{ X\in \E_{\Omega}(V, D) | \dim \ker X(i) \geq 1\}$ and $X(i)$ is defined in (\ref{xi}). 
\end{lem}

Another way of stating (\ref{b}) is 
\begin{equation}
\label{c}
\Supp (\Phi_{\Omega}^{\Omega_i}K) \cap (  \E_{\Omega_i}(V, D) \backslash \E_{\Omega_i, i, \geq 1}(V, D))=\mbox{\O}, \quad \mbox{for some $i$ in $I$}.
\end{equation}

Let $\V_{V, D} $ (resp. $\V_V$) be the localization of $\Q_{V, D}$ with respect to the subcategory $\N_{V, D}$ (resp. $\N_V$). The above analysis produces the following commutative diagram of functors
\begin{equation}
\label{cat-SES}
\begin{CD}
\N_V @>\iota>> \Q_V @>Q>> \V_V\\
@VVV @V\pi^*[\nu d_{\Omega}]VV @VVV\\
\N_{V, D} @>\iota>> \Q_{V, D} @>Q>> \V_{V, D},
\end{CD}
\end{equation}
where the  $\iota$'s are natural embedding, the $Q$'s are localization functors, 
and the unexplained vertical maps are induced from $\pi^*[\nu d_{\Omega}]$.

\begin{rem}
(i) The condition (\ref{c}) is exactly the stability condition used in ~\cite{Zheng08} for localizing the category  $\Q_{V, D}$ 
to get the module $V_{\lambda}$ for a chosen dominant weight $\lambda$ such that $(i, \lambda)=d_i$ for any $i\in I$.

(ii) Note that the rows in the  diagram (\ref{cat-SES}) is exactly the categorical version 
of the short exact sequence (\ref{SES-A-form}).

(iv) As I.B. Frenkel pointed out, the sequence (\ref{SES-A-form}) is the first step of   the BGG resolution of the module $V_{\lambda}$  a $q$-analogue of the resolution in \cite{BGG75}. It is very interesting to investigate  the possibility of naturally lifting the BGG solution to the categorical level.
\end{rem}

\subsection{Stability conditions for $\Q_{V, D}$} 
\label{stable-1}
In the section, we will use the notations in Section ~\ref{Lusztig} freely. 
By ~\cite[13]{Lusztig91}, we have 
\begin{equation}
\label{ss-1}
\SSS(K) \subseteq \mbf  \Lambda_{V\oplus D}, \quad  \forall K\in \Q_{V, D}.
\end{equation}
Moreover,  by using ~\cite[Theorem 13.3]{Lusztig91}, 
\begin{equation}
\label{ss-2}
\SSS(L_{(\mbf{i, a})\cdot d})=\SSS(L_{(\mbf{i, a})} \cdot L_d) \subseteq \{p=0\} \cap\mbf \Lambda_{V\oplus D}.
\end{equation}
If $K\in \Q_{V, D}$ is simple up to a shift, then
 $K$ is a direct summand of the semisimple complex of the form
$L_{(\mbf{i, a})} \cdot L_d$. 
Thus,  by (\ref{ss-1}) and (\ref{ss-2}), 
\[
\SSS(K) \subseteq  \mbf L_{V, D}, \quad \forall K\in \Q_{V, D}.
\]

Let $\M_{V, D}$ be the full subcategory of $\Q_{V, D}$ consisting of objects $K$ such that
\begin{equation}
\label{global}
\SSS(K) \cap \mbf  L^s_{V, D}=\mbox{\O}.
\end{equation}
Since the condition (\ref{global}) does not involve any orientation of the graph $\Gamma$ (or $\tilde \Gamma$), 
we call it a $global$ $condition$, while (\ref{c}) is called a $local$ $condition$.

\begin{prop}
\label{N=M}
We have $\N_{V, D} =\M_{V, D}$.
\end{prop}

\begin{proof}
It is clear from (\ref{b}) or (\ref{c}) that if $K\in \N_{V, D}$,  then $\SSS(K) \cap  \mbf L^s_{V, D}=\mbox{\O}$.  So we have
$\N_{V,D}\subseteq \M_{V, D}$.

By Theorem 6.2.2 in  ~\cite{KS97}, for each simple perverse sheaf  $K$ in $\Q_{V, D}$, one can associate an irreducible component, say $Y_K$, such that  
\[
Y_K \subset \SSS(K) \subset Y_K\cup \cup_{K': \e_i(K') \geq  \e_i(K) } Y_{K'},
\]
for all $i\in I$. 
Moreover, if $K\neq K'$, then $Y_K \neq Y_{K'}$.  

Note that the inequality is a strictly inequality in ~\cite{KS97}, which is a typo.  See Remark 4.27 in ~\cite{Sch09}.  For a proof of this inequality, see ~\cite{K07}.

Thus the assignment $K\mapsto Y_K$ defines an injective map
\[
\phi: S_1 \hookrightarrow S_2, 
\]
where $S_1$ is 
 the set of isomorphism classes of simple perverse sheaves in $\N_{V, D}$ and 
$S_2$ is 
the set  of irreducible components in $ \mbf L_{V, D} $ disjoint from $\mbf L^s_{V, D}$.
Observe that  the  sets  $S_1$ and $S_2$ have the same number of elements, equal to 
$\dim T_{\lambda, \nu}$ due to ~\cite{Lusztig93} and ~\cite{L00a}.  So $\phi$ is bijective. 
This implies that $\N_{V, D} = \M_{V, D}$.  Otherwise, if $K\in \M_{V, D}\backslash \N_{V, D}$ is simple, then $Y_K\cap L_{V, D}^s=\mbox{\O}$ by definition. 
Since $\phi$ is bijective, there is a $K'\in \N_{V, D}$ such that $Y_{K'} = Y_K$. This contradicts with the fact that $K\neq K'$ implies that $Y_K\neq Y_{K'}$.
Proposition follows.
\end{proof}

From Proposition ~\ref{N=M}, we have the following corollary.

\begin{cor}
Assume that  $K\in \Q_{V, D}$. $Q(K)\not = 0$ if and only if $\SSS(K) \cap\mbf  L^s_{V, D}\not = \mbox{\O}$.
\end{cor}

Let $\mathcal P_{V, D}$ be the set of all isomorphism classes of simple perverse sheaves in $\Q_{V, D}$. 
Let $\mathcal P_{V, D}^s$ be the subset of $\mathcal P_{V, D}$ consisting of all elements not in $\N_{V, D}$.
Then we have the following corollary.

\begin{cor} 
\label{def-ps}
For any $K\in \mathcal P_{V, D}$, the following statements are equivalent.
\begin{enumerate}
\item $K\in \mathcal P_{V, D}^s$.  \hspace{.5cm} $(2)$  $Q(K) \neq 0$. \hspace{.5cm} $(3)$ $\SSS(K) \cap \mbf L^s_{V, D}\neq \mbox{\O}$.
\item[(4)] $\Supp (\Phi_{\Omega}^{\Omega_i} (K)) \cap \E_{\Omega_i, i, 0}(V, D) $ is open dense in $\Supp(\Phi_{\Omega}^{\Omega_i} (K))$, where 
$ \E_{\Omega_i, i, 0}(V, D)=\{ X\in \E_{\Omega_i} (V, D) | \dim \ker X(i)=0\}$ for any $i\in I$.
\end{enumerate}
\end{cor}

\begin{rem}
The results in the section are the combination of the work ~\cite{Lusztig91} and ~\cite{KS97}.  Results in ~\cite[11]{Nakajima94} are closely related to the results in this section.
\end{rem}

\section{Geometric study of   tensor product $M_{\lambda^2}\otimes V_{\lambda^1}$}

\label{twocopy}

\subsection{Tensor product complexes} 
\label{tensorproductcomplexes}
In this section, we fix three  elements $\lambda$,  $\lambda^1, \lambda^2$ in $\mbf X^+$ such that 
\[
\lambda^1+\lambda^2=\lambda.
\]
This matches with the fixed decomposition $D=D^2\oplus D^1$ in Section ~\ref{twocopytensor} by assuming that 
$(i, \lambda) =\dim D_i$ and $(i, \lambda^a) =\dim D^a_i$ for any $ i\in I$ and $a=1, 2$.
Let $d^a$ be the dimension vector of $D^a$ for $a=1, 2$.

In this section, we will consider the compositions
\[
\underline{\mbf a} := (\mbf i^1, \mbf a^1) \cdot d^1\cdot (\mbf i^2, \mbf a^2) \cdot d^2
\]
such that $\sum_{k=1}^{m^1} a^1_ki_k^1 + \sum_{k=1}^{m^2} a^2_k i_k^2 =\nu$ where $\nu$ is the dimension vector of $V$ in the space
$\E_{\Omega}(V, D)$.
We can define the map 
\[
\pi^l_{\underline{\mbf a}}: \tF_{\underline{\mbf a}}\to \E_{\Omega}(V, D),
\]
in exactly the same manner as the map $\pi_{\mbf {i, a}}$ defined in (\ref{proper}).

Fix a partial flag $\check D^{\bullet}=(\check D^0=D\supseteq \check D^1\supseteq 0)$ such that $\dim \check D^0/\check D^1=d^1$. 
In other words, we fix  a subspace  $\check D^1$  of $D$ of dimension $d^2$. 
Let 
$\tE_{\underline{\mbf a}}$ be the subvariety  of $\tF_{\underline{\mbf a}}$ 
consisting of all pairs whose flags incident to $D$ is $\check D^{\bullet}$.

The restriction of $\pi^l_{\underline{\mbf a}}$ to $\tE_{\underline{\mbf a}}$ is denoted by
\[
\pi_{\underline{\mbf a}}: \tE_{\underline{\mbf a}}\to \E_{\Omega}(V, D),
\]
in exactly the same manner as the map $\pi_{(\mbf {i, a})}$ defined in (\ref{proper}). Note that $\tE_{\underline{\mbf a}}$ is again a smooth irreducible variety. 
So we have the following semisimple complex
\begin{align}
\label{La}
L_{\underline{\mbf a}} :=\pi_{\underline{\mbf a}!}(\mbb C_{\tE_{\underline{\mbf a}}})[\dim \tE_{\underline{\mbf a}}]. 
\end{align}
Note that when $\mbf a^2=0$, $L_{\underline{\mbf a}}$ is the same as $L_{(\mbf{i, a})\cdot d}$.

Similar to $\Q_{V, D}$, 
let $\Q_{V, D^{\bullet}}$ be the full `semisimple'  subcategory of  $\D(\E_{\Omega}(V, D))$ whose simple objects are isomorphic to those appeared in 
$L_{\underline{\mbf a }}$ for various compositions $\underline{\mbf a}$.

Similar to $\N_{V, D}$, we define $\N_{V, D^{\bullet}}$ to be the full subcategory of 
$\Q_{V, D^{\bullet}}$ generated by the simple objects $K$ satisfying the local stability condition (\ref{b}).

Similar to $\V_{V, D}$, we define $\V_{V, D^{\bullet}}$ to be the localization of the category $\Q_{V, D^{\bullet}}$ with respect to $\N_{V, D^{\bullet}}$.
Altogether,  we have the following exact sequence of categories
\begin{equation}
\N_{V, D^{\bullet}} \to \Q_{V, D^{\bullet}} \to \V_{V, D^{\bullet}}.
\end{equation}
Since $L_{(\mbf {i, a})\cdot d}$ is a special case of the complex $L_{\underline{\mbf a}}$, we have the following commutative diagram
\begin{equation}
\label{M-tensor}
\begin{CD}
\N_{V, D} @>>> \Q_{V, D} @>>> \V_{V, D}\\
@VVV @VVV @VVV\\
\N_{V, D^{\bullet}} @>>> \Q_{V, D^{\bullet}} @>>>\V_{V, D^{\bullet}},
\end{CD}
\end{equation}
where the top level is from (\ref{cat-SES}) and the vertical maps are inclusions. While the diagram (\ref{cat-SES}) indicates that we go from 
non framed situation to framed situation. Here, this diagram (\ref{M-tensor}) indicates that we go from the Verma module to the tensor product of a Verma module with a simple module.

\subsection{Structures on $\mathcal P_{V, D^{\bullet}}$}

Let $_i\Omega$ be an orientation of $\tilde \Gamma$ such that $i$ is a sink, i.e., all arrows, h, adjacent to $i$ has $h''=i$.
Let $_{i, n}\E_{_i\Omega}(V, D)$ be the locally closed subvariety of $\E_{_i\Omega}(V, D)$ consisting of all elements $X$ such that 
$\dim V_i/ (i)X= n$.  Here $(i) X$ is defined in (\ref{xi}). We also set 
\[
_{i, \geq n}\E_{_i\Omega}(V, D)=\sqcup\, _{i, n'} \E_{_i\Omega}(V, D),
\] 
where the union runs over all $n'$ such that $n'\geq n$, 
which is a closed subvariety of $\E_{_i\Omega} (V, D)$.

Let $\mathcal P_{V, D^{\bullet}}$ be the set of isomorphic classes of simple perverse sheaves in $\Q_{V, D^{\bullet}}$ defined in Section ~\ref{tensorproductcomplexes}.
Let $_{i, n}\mathcal P_{V, D^{\bullet}}$ be the subset of $\mathcal P_{V, D^{\bullet}}$ consisting of all objects $K$ such that 
\[
\Supp ( \Phi_{\Omega}^{_i\Omega} (K) ) \subset \, _{i, \geq n} \E_{_i\Omega}(V, D)
\quad  \mbox{and}\quad 
\Supp ( \Phi_{\Omega}^{_i\Omega} (K) ) \cap  \, _{i, n} \E_{_i \Omega}(V, D)\neq \mbox{\O}
\]
where $\Phi_{\Omega}^{_i\Omega} $ is the Fourier transform.
We set
\begin{equation}
\label{e_i}
\e_i(K)  = n, \quad \mbox{if} \; K\in \!\;  _{i, n}\mathcal P_{V, D^{\bullet}}.
\end{equation}

\begin{lem} 
\label{strut-lem}
For any $K\in \, _{i, n}\mathcal P_{V, D^{\bullet}}$, there exists a unique $\bar K\in \, _{i, 0} \mathcal P_{W, D^{\bullet}}$, with $\dim W = \dim V-ni$, such that the following statements hold.

\begin{enumerate}
\item   $_i\mathcal R^{(n)} (K) =\bar K\oplus \bar K'$, where $\bar K'$ consists of simple perverse sheaves, with possible shifts, in $_{i, \geq 1} \mathcal P_{W, D^{\bullet}}$ and
            $_i\R^{(n)}$ is  defined in Section ~\ref{special}, (\ref{iR}).

\item  $\mathcal F_i^{(n)} (\bar K) =K \oplus K'$, where $K'$ consists of simple perverse sheaves, with possible shifts, in $_{i, \geq n+1}\mathcal P_{V,D^{\bullet}}$ and 
           $\F_i^{(n)}$ is defined in Section ~\ref{special}, (\ref{iR}).

\item The assignment $K\mapsto \bar K$ defines a bijection $_i\tilde{ \mathcal R}^{(n)}:\, _{i, n} \mathcal P_{V, D^{\bullet}} \to \,  _{i, 0} \mathcal P_{W,D^{\bullet}}$.

\item The assignment $\bar K\mapsto K$ defines a bijection 
          $\tilde {\mathcal F}_i^{(n)}: \, _{i, 0}  \mathcal  P_{W, D^{\bullet}} \to \, _{i, n}\mathcal P_{V, D^{\bullet}}$, inverse to $_i\tilde{\R}^{(n)}$.
\end{enumerate}
\end{lem}

The proof is  similar to the proofs of Lemma 3.2.10 in ~\cite{Zheng08} and  Lemma 6.4 in ~\cite{Lusztig91}.

\begin{prop} 
\label{B-L}
For any $K\in \Q_{V, D^{\bullet}}$, there exist complexes $M_1, \cdots, M_k$ and $N_1,\cdots, N_l$  of the form $L_{\underline{\mbf  a}}$ for some $\underline {\mbf a}$, up to some shift, such that 
\[
K \oplus M_1\oplus \cdots \oplus M_k =N_1\oplus \cdots \oplus N_l.
\]
\end{prop}
The proof is  similar to the proofs of Proposition 3.2.6 in ~\cite{Zheng08} and Proposition 7.3 in ~\cite{Lusztig91}.

\subsection{Singular supports of the objects in $\mathcal P_{V, D^{\bullet}}$}

\label{singular}

Let $\tilde X_{\underline{\mbf a}}$ be the variety of all pairs $(X, F^{\bullet})$, where $X\in \mbf \Lambda_{V, D}$ and $F^{\bullet}$ is a flag of type $\underline{\mbf a}$,
such that $F^{\bullet}$ is $X$-invariant. Then we have a natural projection
\[
\tilde X_{\underline{\mbf a}} \to \mbf  \Lambda_{V, D}.
\]
We denote by $\tilde Y_{\underline{\mbf a}}$ the image of $\tilde X_{\underline{\mbf a}}$ under this projection.
Similar to the proof of Theorem 13.3 and Corollary 13.6 in ~\cite{Lusztig91}, one can show the following proposition. 

\begin{prop}
\label{Tensor-singular-1}
$\SSS(L_{\underline{ \mbf a}}) \subseteq \tilde{Y}_{\underline{\mbf a}} \subseteq  \mbf \Lambda_{V, D}$.
\end{prop}

Given any pair $(X, F^{\bullet})$ in $\tilde X_{\underline{\mbf a}}$, we have that both $F^{m^1}$ and $F^{m^1+1}$ are $X=(x, p, q)$-invariant. 
 Write $F^{m^1} = U\oplus D^0$ and $F^{m^1+1}= U\oplus D^1$, we have $p(D^0) \subseteq U$ and  $q(U) \subseteq D^1$. This implies that 
 $p(D^0) \subseteq U \subseteq q^{-1}(D^1)$.  Since $U$ is $x$-invariant, we have  
 $\overline {p(D^0)} \subseteq U \subseteq \underline {q^{-1}(D^1)}$.
Since $F^{m^1+m^2+1}$ is $X$-invariant, we have immediately $\overline{ p(D^1)}=0$. 
So given any pair $(X, F^{\bullet}) \in \tilde X_{\underline{\mbf a}}$, the element $X$ satisfies the condition (\ref{tensor-condition}).
 In other words, $\tilde Y_{\underline{\mbf a}}\subseteq \mbf  \Pi_{V, D^{\bullet}}$.
By combining with Proposition ~\ref{Tensor-singular-1}, we have 
 
 \begin{prop}
 \label{Tensor-singular-2}
 $\SSS(L_{\underline{\mbf a}})\subseteq \tilde Y_{\underline{\mbf a}} \subseteq \mbf \Pi_{V, D^{\bullet}}\subseteq \mbf  \Lambda_{V, D}$.
 \end{prop}

\subsection{The map $Y_{\bullet}$}

\label{mapY}

First, we assume that $D^2=0$ in the set up of Section ~\ref{tensorproductcomplexes}.
In this case, the complex $L_{\underline{\mbf a}}$ defined in (\ref{La}) is a special case of the complex $L_{\mbf{i, a}}$ defined in (\ref{Lia}) for the graph $\tilde \Gamma$. 
Also, $\mbf \Pi_{V, D^{\bullet}}\subseteq \mbf \Lambda_{V\oplus D}$.
Thus, we can use Theorem 6.2.2 (2) in ~\cite{KS97}, ~\cite{K07} and Remark 4.27 in ~\cite{Sch09} to define a  map
\begin{align}
\label{Y}
Y_{\bullet}: \mathcal P_{V, D^{\bullet}} \to \Irr \mbf \Pi_{V, D^{\bullet}}, \quad K \mapsto  Y_K,
\end{align}
such that
\begin{align*}
&Y_K\subseteq \SSS(K) \subseteq Y_K \cup \cup_{K': \e_i(K')\geq \e_i(K)} Y_{K'},\\
&\e_i(K) =\e_i(Y_K), \quad \forall i\in I,\\
& K\neq K'  \quad\mbox{implies}\quad Y_K\neq Y_{K'}.
\end{align*}
The last condition means that $Y_{\bullet}$ is an injective map. Moreover, $Y_{\bullet} $ is surjective hence bijective. 

To show that $Y_{\bullet}$ is bijective, it is reduced to show that the two sets $\mathcal P_{V, D^{\bullet}}$ and $\Irr \mbf \Pi_{V, D^{\bullet}}$ are of the same size.
This can be argued as follows. Let $\mbf V_1$ be the space over $\mbb C$ spanned by the elements in $\mathcal P_{V, D^{\bullet}}$. 
Let $\mbf V_2$ be the space over $\mbb C$ spanned by the complexes $L_{\underline{\mbf a}}$ for various $\underline{\mbf a}$. 
Then, modulo specializing the shift functor to $1$, we have by Proposition ~\ref{B-L} that  
\[
\mbf V_1 \simeq \mbf  V_2
\]
as vector spaces over $\mbb C$.

On the other hand, let $\mbf W_2$ be the space spanned by the constructible functions $L'_{\underline{\mbf a}}$ on $\mbf \Pi_{V, D^{\bullet}}$ 
defined in a similar way as $L_{\underline{\mbf a}}$.   Let  $\mbf W_1$ be the vector space over $\mbb C$ spanned by the semicanonical basis elements
$f_Y$ defined in ~\cite{L00a} such that $Y\subseteq \mbf \Pi_{V, D^{\bullet}}$.  Then we have 
(We refer the interested reader to the paper ~\cite{L00a} and the reference therein for the precise definitions of $L'_{\underline{\mbf a}}$ and $f_Y$.)
\[
\mbf W_1 = \mbf W_2.
\]

Indeed, it is clear that $\mbf W_2 \subseteq \mbf W_1$. Suppose that $f_Y\in \mbf W_1$ satisfies that $\e_i(Y) =0$ for any $i\in I$. It is clear that 
$f_Y = L_d' \cdot f_{\bar Y}$ for some $\bar Y\in \mbf \Lambda_V$ where $L_d'$ is the linear map defined in analog with the functor $L_d$.
From this, we see that $f_Y\in \mbf W_2$. Suppose that $f_Y\in\mbf W_1$ such that $r=\e_i(Y) >0$ for some $i\in I$. 
We prove by induction that $f_Y\in \mbf W_2$. We assume that $f_{Y'}\in \mbf W_2$ if $Y\subseteq \mbf \Pi_{V', D^{\bullet}}$ for any proper subspace $V'\subset V$ and
$f_Z\in \mbf W_2$ if $\e_i(Z)> \e_i(Y)$.

By ~\cite[2.9 (b)]{L00a}, we have an irreducible component $Y'$ such that 
\[
f_i^{r} \cdot f_{Y'} = f_Y + \sum_{Z: \e_i(Z)>\e_i(Y) } c_{Y, Z} f_Z, \quad \mbox{for some} \quad c_{Y, Z} \in \mbb Z,
\]
where $f_i^{r}$ is defined in analog with the functor $\F_i^{(r)}$. 
Since $Y\subseteq \mbf \Pi_{V, D^{\bullet}}$, we have $Y'\subseteq \mbf \Pi_{V', D^{\bullet}}$ for some $V'$. Hence, $ Z\subseteq \mbf \Pi_{V, D^{\bullet}}$. 
By the above identity and induction assumption, we see that $f_Y\in \mbf W_2$. So $\mbf W_1\subseteq \mbf W_2$.  We have finished the proof of $\mbf W_1=\mbf W_2$.

By ~\cite{Lusztig91} and ~\cite{L00a}, 
we see that 
\[
\mbf V_2 = \mbf W_2
\] 
because they correspond to the same subspace in $U^-$, the space obtained from $\U^-$ by specializing $v$ at $1$.
By summing up the above analysis, we have 
\[
\#\mathcal P_{V, D^{\bullet}} =\dim \mbf  V_1 =\dim  \mbf V_2 =\dim \mbf  W_2 =\dim \mbf W_1=\#\Irr \mbf \Pi_{V, D^{\bullet}}.
\]
This shows that the maps $Y_{\bullet}$ is surjective. Altogether, we have

\begin{lem}
\label{Y-0}
When $D^2=0$, the map $Y_{\bullet}$ is bijective.
\end{lem}

Second, we  assume that $D^2\neq 0$. 
We write $\mathcal P_{V, D^1}$ (resp. $\mbf \Pi_{V, D^1}$) for  $\mathcal P_{V, D^{\bullet}}$ (resp. $\mbf \Pi_{V, D^{\bullet}})$ for the case when  $D^2=0$.
The fully faithful functor $\cdot L_{d^2}$ defines a bijection
\[
L_{d^2}: \mathcal P_{V, D^1} \to \mathcal P_{V, D^{\bullet}},
\]
such that $\e_i(\tilde K) = \e_i (\tilde K\cdot  L_{d^2})$ for any $i\in I$.
Similarly, we can define a bijection between the sets $\Irr \mbf \Pi_{V, D^1}$ and $\Irr \mbf \Pi_{V, D^{\bullet}}$. It is clear from the construction that the two bijections are compatible.
Therefore, we can define a similar map 
$Y_{\bullet}:  \mathcal P_{V, D^{\bullet}} \to \mbf \Pi_{V, D^{\bullet}}$ satisfying the same property as that of (\ref{Y}).

Summing up the above analysis, we have the following proposition.

\begin{prop}
\label{YY}
We have a bijective map $Y_{\bullet}:  \mathcal P_{V, D^{\bullet}} \to \Irr \mbf \Pi_{V, D^{\bullet}}$, $K\mapsto Y_K$ such that 
\begin{align*}
Y_K\subseteq \SSS(K) \subseteq Y_K \cup \cup_{K': \e_i(K')\geq \e_i(K)} Y_{K'}\quad  \mbox{and}\quad 
\e_i(K) =\e_i(Y_K), \quad \forall i\in I.
\end{align*}
\end{prop}

\begin{rem}
We shall present a second proof of Proposition ~\ref{YY} in Section ~\ref{filtration-P}, which can be generalized to the general case in Section \ref{general}.
\end{rem}

For any $K\in \mathcal  P_{V, D^{\bullet}}$ and $i\in I$, we define 
\begin{align*}
& \wt(K) = \lambda -\nu \in \X,\quad  \e_i(K) =n,\quad \mbox{if $K \in \, _{i, n}\mathcal P_{V, D^{\bullet}}$},\quad
 \varphi_i=\e_i(K)+ (i, \wt(K)),
\end{align*}
where $\lambda \in \X$ is a fixed element such that $\lambda(i)=d_i$, for any $i\in I$,  and $\nu\in \X$ is via the imbedding 
$\mbb N[I]\to \X$.
We also define, for any $K\in \mathcal P_{V, D^{\bullet}}$, 
\begin{align*}
\tf_i(K) =\tilde{\mathcal F}_i^{(\e_i(K)+1)}\, _i \tilde{\mathcal R}^{(\e_i(K))}(K) 
\quad\mbox{and}\quad
\te_i(K) =
\begin{cases}
\tilde {\mathcal F}_i^{(\e_i(K)-1)} \, _i \tilde{\R}^{(\e_i(K))} (K) & \mbox{if}\; \e_i(K)>0,\\
 0 &\mbox{otherwise}.
 \end{cases}
\end{align*}

The above maps ($\wt$, $\e_i$, $\varphi_i$, $\te_i$, $\tf_i)_{i\in I}$ define a crystal structure on 
\[
\mathcal P_{D^{\bullet}}=\sqcup_{\nu\in \mbb N[I]} \mathcal P_{V(\nu), D^{\bullet}}.
\]
Moreover, it is clear that 
the crystal $\mathcal P_{D^{\bullet}}$ is generated by the objects $K$ such that $\e_i(K)=0$ for any $i\in I$.

The assignment $K \mapsto Y_K$ in (\ref{YY}) defines a  strict crystal isomorphism 
\[
Y_{\bullet}: \mathcal P_{D^{\bullet}} \to \Irr (D^{\bullet}).
\] 
From this isomorphism and Theorem ~\ref{Pi-crystal}, we have 

\begin{thm}
\label{crystal-P}
The crystal structure  on $\mathcal P_{D^{\bullet}}$ is isomorphic to the crystal structure of $B(\lambda^1)\otimes B(\lambda^2,\infty)$ via $\psi Y_{\bullet}$ where $\psi$ is defined in (\ref{psi}) .
\end{thm}

\subsection{A filtration on $\mathcal P_{V, D^{\bullet}}$}
\label{filtration-P}

Let 
\[
\Xi_{V, D^{\bullet}}=\{ K\in \mathcal P_{V, D^{\bullet}} | \; \e_i(K) =0, \quad \forall i\in I\}
\quad \mbox{and} \quad \Xi_{D^{\bullet}}=\sqcup_{\nu\in \mbb N[I]} \Xi_{V(\nu), D^{\bullet}}.
\]

\begin{lem}
\label{xi-1}
We have $K\in \Xi_{V, D^{\bullet}}$ if and only if $K$ is of the form $L_{d^1} K_b L_{d^2}$ where $K_b$ is the simple perverse sheaf corresponding to an element $b$ in $B(\lambda^1)_{\nu}$ with $\nu =\dim V$.
\end{lem}

\begin{proof}
Suppose that $K\in \Xi_{V, D^{\bullet}}$. Then it is clear that $K$ is a direct summand of a certain complex $L_{d^1 (\mbf i, a) d^2}=L_{d^1}\cdot L_{(\mbf i, \mbf a)} \cdot L_{d^2}$, up to a shift.  Since $L_{d^1}\cdot$ and $\cdot L_{d^2}$ are fully faithful functors, we see that $K$ has to be of the form $L_{d^1} K_b L_{d^2}$ for some simple perverse sheaf
in $\mathcal P_V$.
If $K_b$ is a direct summand of the complex $L_{(\mbf i, \mbf a)(i, d_i^1+1)}$, then it is clear that $\e_i(K) >0$. 
This shows that $K_b$ is a simple perverse sheaf such that $b\in B(\lambda^1)_{\nu}$.  
It is clear that if $K=L_{d^1} K_b L_{d^2}$ such that $b\in B(\lambda^1)_{\nu}$ then $K\in \Xi_{V, D^{\bullet}}$.
\end{proof}

Note that in general, the set $\Xi_{D^{\bullet}}$ has infinitely many elements. 
Let us order the elements in $\Xi_{D^{\bullet}}$ in a way, say 
\[
\xi_1, \xi_2, \cdots, \xi_n,\cdots ,
\]
such that if $\xi_m\in \Xi_{V, D^{\bullet}}$ and $\xi_n\in \Xi_{W,D^{\bullet}}$ with $W$ a proper subspace of $V$, then $n < m$. 

Let $\Q_{V, D^{\bullet}}^{\leq n} $ be the full subcategory of $\Q_{V, D^{\bullet}}$ whose simple objects are direct summands of the complexes
$L_{(\mbf i, \mbf a)}\cdot \xi_m[z]$ in $\Q_{V, D^{\bullet}}$, for $z\in \mbb Z$ and  $1\leq m\leq n$. Then we have a filtration of $\Q_{V, D^{\bullet}}$: 
\begin{equation}
\label{filtration-Q}
\Q^{\leq 1}_{V, D^{\bullet}}\subset \Q^{\leq 2}_{V, D^{\bullet}} \subset \cdots \subset \Q^{\leq n}_{V, D^{\bullet}} \subset \cdots.
\end{equation}
Note that $\Q^{\leq n}_{V, D^{\bullet}}=\Q^{\leq m}_{V, D^{\bullet}}=\Q_{V, D^{\bullet}}$ for large enough $m$ and $n$. 

Similarly, we have a filtration for the simple objects in $\Q_{V, D^{\bullet}}$:
\[
\mathcal P^{\leq 1}_{V, D^{\bullet}}\subset \mathcal P^{\leq 2}_{V, D^{\bullet}} \subset \cdots \subset \mathcal P^{\leq n}_{V, D^{\bullet}} \subset \cdots,
\]
where $\mathcal P^{\leq n}_{V, D^{\bullet}}$ are subsets of $\mathcal P_{V, D^{\bullet}}$ whose element appears in $\Q^{\leq n}_{V, D^{\bullet}}$. We set
\[
\mathcal P^n_{V, D^{\bullet}}= \mathcal P^{\leq n}_{V, D^{\bullet}} \backslash \mathcal P^{\leq n-1}_{V, D^{\bullet}}, \quad \forall n\in \mbb N.
\]
(The filtration comes from the one in ~\cite{BGG71}, ~\cite{BGG75}, ~\cite{BGG76} and ~\cite{BG80}.)

\begin{lem}
\label{n=0}
We have $\xi_n\in  \mathcal P^n_{V, D^{\bullet}}$.
\end{lem}

\begin{proof}
Assume that $\xi_n \in \Q^{\leq n-1}_{V, D^{\bullet}}$. Since $\e_i(\xi_n) =0$ for any $i\in I$, $\xi_n$ can only be expressed as a direct sum whose simple summands are from the set
$\{\xi_1, \cdots,  \xi_{n-1}\}$. This contradicts with the fact that the set  $\{\xi_1,\cdots, \xi_{n-1}, \xi_n\}$ is linearly independent, because it is a subset of  the canonical basis $ B(\lambda^1)$ by Lemma ~\ref{xi-1}.  Therefore, $\xi_n\in \mathcal P^n_{V, D^{\bullet}}$.
\end{proof}

Recall from ~\cite[Theorem 6.2.2]{KS97} that there exists a crystal isomorphism 
\[
Y_{\bullet} : \sqcup_{V} \mathcal P_V \to \Irr \mbf \Lambda_V,
\]
where $V$ runs over  a set of representatives of the isomorphism classes of the $I$-graded vector spaces, 
such that 
\[
Y_K \subseteq \SSS(K) \subseteq Y_K \cup \cup_{\e_i(Y_K)\geq \e_i(Y_{K'}) } Y_{K'}, \quad \forall i\in I.
\]
Let $Y_{\xi_n}$ be the irreducible component in $\Irr \mbf \Pi_{V, D^{\bullet}}$ obtained from $Y_{K_b}$ if $\xi_n = L_{d^1} K_{b} L_{d^2}$. 
Let $\mbf \Pi_{V, D^{\bullet}}^n$ be the subset in $\Irr \mbf  \Pi_{V, D^{\bullet}}$ defined in a similar manner as $\mathcal P^n_{V, D^{\bullet}}$ to $\mathcal P_{V, D^{\bullet}}$.
In other words, $\mbf \Pi_{V, D^{\bullet}}$ is the crystal generated by the element $Y_{\xi_n}$.

Let  $\pi_a'$ be  the similar morphisms to
$\pi_a$ in (\ref{induction-diag}) for $a=1,2, 3$ with the varieties `$\mbf E$' replaced by the varieties `$\mbf \Pi$'. 
For any irreducible component $\bar Y\in \Irr \mbf \Lambda_T$, we set 
\[
\bar Y \cdot Y_{\xi_n}= \pi_3' \pi'_2(\pi'_1)^{-1} (\bar Y\times Y_{\xi_n})
\] 
to be the closed  subvariety of $\mbf \Pi_{V, D^{\bullet}}$, where we assume that $Y_{\xi_n} \in \Irr \mbf \Pi_{W, D^{\bullet}}$ and $V=T\oplus W$. 
Let $Y_{\xi_n}^0$ be the open subvariety of $Y_{\xi_n}$ consisting of all points $X$ such that $\e_j(X) = 0$ for any $j\in I$. 
We claim  that 
\begin{itemize}

\item the closure, $Y$,  of $\bar Y \cdot Y_{\xi_n}^0$ is an irreducible component in  $\mbf \Pi_{V, D^{\bullet}}$ such that $\e_j (Y) = \e_j(\bar Y)$ for any $j\in I$.  
\end{itemize}

Indeed, the restriction of $\pi_1'$ to $(\pi_1')^{-1}(\bar Y\times Y_{\xi_n}^0)$ is smooth with connected fiber. This can be proved in a similar manner as the proof of 
Lemma ~\ref{dimrho} due to the fact that the condition that $X(i)$ is injective is due to the condition that $(i)X$ is surjective. 
Moreover, the restriction of $\pi_3'$ to $\pi'_2(\pi_1)'^{-1} (\bar Y\cdot Y_{\xi_n}^0)$ is injective. This is due to the fact that the space $U$ such that $D\oplus U$ is $X$-stable for any $X=(x, p, q)\in \pi'_2(\pi_1)'^{-1} (\bar Y\cdot Y_{\xi_n}^0)$ is $\overline{p(D)}$ by the definition of $Y_{\xi_n}^0$.
The above analysis implies that $\bar Y\cdot Y_{\xi_n}$ is irreducible, of the desired dimension  and $\e_j(\bar Y\times Y_{\xi_n})=\e_j(\bar Y)$ for any $j\in I$. The claim follows. 

It is clear that 
\[
Y\subseteq \bar Y\cdot Y_{\xi_n} \subset Y\cup \cup_{\e_i(Y')\geq \e_i(Y)} Y', \quad \forall i\in I. 
\]

\begin{prop}
\label{finer-Pi}
 The assignment $\bar Y\mapsto Y$, where $Y$ is the closure of $\bar Y \cdot Y_{\xi_n}^0$,  
 defines a bijection 
 \[
\gamma:  \Irr \mbf \Lambda_T \to \Irr \mbf \Pi_{V, D^{\bullet}}^n,
 \]  
 which is compatible with the actions $\te_i^r$ and $\tf_i^r$  for any $i\in I$ and $r\in \mbb N$. 
\end{prop}

\begin{proof}
It is clear that $\gamma$ is injective. $\gamma$ is surjective follows from Theorem ~\ref{Pi-crystal}. It can also be proved by induction with respect to the dimension of $T$.
The compatibility of $\gamma$ with the actions $\te_i^r$ and $\tf_i^r$ is reduced to the compatibility of the diagram  (\ref{finer-action}) and the  diagram similar  to (\ref{induction-diag}).
\end{proof}

\begin{prop}
\label{filtration-prop}
Suppose that $\xi_n\in \Xi_{W, D^{\bullet}}$. 
Given any $\bar K\in \mathcal P_T$, there is a unique $K\in \mathcal P^n_{V, D^{\bullet}}$, with $V\simeq T\oplus W$, such that
\begin{equation}
\label{filtration-prop-a}
\begin{split}
& \bar K\cdot \xi_n= K \oplus M, \quad \mbox{where} \;  M\in \mathcal \Q^{\leq n}_{V, D^{\bullet}};\\
 &\e_j (\bar K) =\e_j (K), \quad \e_j (K) \leq   \e_j (M), \quad  \forall j\in I;\quad 
  \e_i(K) < \e_i(M), \quad \mbox{for some}\; i\in I;\\
 & \gamma(Y_{\bar K}) \subseteq \SSS(K) \subseteq \gamma(Y_{\bar K}) \cup \cup_{\e_j(Y') \geq  \e_j(K)} Y', \quad \forall j\in I.
\end{split}
\end{equation}
Moreover, all elements in $\mathcal P^n_{V, D^{\bullet}}$ are obtained in this way.
The assignment $\bar K\mapsto K$ in (\ref{filtration-prop-a}) defines a bijection
$\mathcal P_T \to \mathcal P^n_{T\oplus W, D^{\bullet}}$.
\end{prop}

\begin{proof}
We shall prove the statement by induction with respect to the dimension of $T$.
The statement is clear when $T=0$ by Lemma ~\ref{n=0} and the choice of $Y_{\xi_n}$. Suppose that the statement holds for any proper subspaces in a non zero vector space $T$.  Then there exists a vertex $i\in I$ such that 
$\e_{i}(\bar K) > 0$.  Set $r=\e_i(\bar K) >0$. Then, by Lemma \ref{strut-lem}, there exists a simple perverse sheaf $\bar L\in \mathcal P_{\bar T}$ for some  subspace $\bar T$ in $T$  with $\e_i(\bar L)=0$ such that 
\[
\F_i^{(r)}  \cdot \bar L = \bar K  \oplus \bar M, \quad \mbox{where}\; \e_i(\bar M)> r.
\]
By induction,  we have 
\[
\bar L\cdot  \xi_n = L \oplus N, \quad \mbox{where}\; L\in \mathcal P^{n}_{\bar V, D^{\bullet}} \; \mbox{and}\; N\in \Q^{\leq  n}_{\bar V, D^{\bullet}},
\]
for $\bar V$ a certain proper subspace in $V$. Moreover, the complexes $L$ and $N$ satisfy  $\e_i(L) =0$ and $\e_i(N)>0$ and $\e_j(N) \geq \e_j(L)$ for any $j\in I$. From this and by Lemma ~\ref{strut-lem}, we see that there is a unique simple perverse sheaf $K$ in
$\F_i^{(r)} \cdot \bar L\cdot \xi_n$ such that $\e_i(K) =r$ and that $K$ is a direct summand of $\F_i^{(r)}\cdot L$.
Observe that  
\[
\F_i^{(r)} L \oplus \F_i^{(r)} N =\F_i^{(r)} \cdot \bar L\cdot \xi_n= \bar K\cdot \xi_n \oplus \bar M\cdot \xi_n,
\]
 and  $\e_i(\bar M\cdot \xi_n), \e_i(\F_i^{(r)} N) > r$.
We see that $K$ has to be a direct summand  in $\bar K\cdot \xi_n$, i.e., 
\[
\bar K \cdot \xi_n = K\oplus M, \quad \mbox{for some $M \in\mathcal \Q^{\leq n}_{V, D^{\bullet}}$ such that } \e_i(M) > r.
\]

We are left to show that $\e_j(\bar K) = \e_j(K)$  for any $j\in I$ and the claim on their  singular supports. (This will automatically imply that $\e_j(K) \leq \e_j(M)$ for any $j\in I$.) 
By the induction assumption, we have 
\[
\gamma( Y_{\bar L}) \subseteq \SSS(L).
\]
This implies that 
\[
\gamma(Y_{\bar K})  = \tf_i^r  \gamma(Y_{\bar L}) \subseteq \SSS( \F_i^{(r)} \cdot L),
\]
where the equality is due to Proposition ~\ref{finer-Pi}.
Now the complex $K$ is the only one in $\F_i^{(r)} \cdot L$ such that $\e_i(K) = r$ and  the evaluations of $\e_i$ at  all other simple summands have values  strictly larger than $r$. So
\[
\gamma(Y_{\bar K})   \subseteq \SSS(K). 
\]
This implies that $\e_j(K) = \e_j( Y(\bar K)) =\e_j (\bar K)$ for any $j\in I$. 

Since $L\in \mathcal P^n_{\bar V, D^{\bullet}}$,  so is $K$.  Otherwise, it will lead to a contradiction by Lemma ~\ref{strut-lem}.
The statement (\ref{filtration-prop-a}) follows.

Now we show that all elements $K$ in  $\mathcal P^n_{V, D^{\bullet}}$ can be obtained in the above way. If $\e_i(K) =0$ for all $i\in I$, then $K=\xi_n$.  The proof is trivial. 
If there is an $i$ such that $\e_i(K) \neq 0$, then we can use a similar argument as above to show that there is a complex $\bar K\in \mathcal P_T$ for some $T$ such that condition 
(\ref{filtration-prop-a}) holds.
\end{proof}

By combining Propositions ~\ref{finer-Pi} and ~\ref{filtration-prop}, we have the following commutative diagram
\[
\begin{CD}
\mathcal P_T @>>> \mathcal P_{V, D^{\bullet}}^n \\
@VVV @VVY_{\bullet} V\\
\mbf \Pi_V  @>>> \mbf \Pi_{V, D^{\bullet}}^n,
\end{CD}
\]
where $Y_{\bullet} (K) = \gamma (Y_{\bar K})$. Since all three unnamed maps are bijective, we see that $Y_{\bullet}$ is bijective. This is a second proof of Proposition ~\ref{YY}.

\subsection{Grothendieck group $\mathcal K_{D^{\bullet}}$} 

Let $_{\mbb A} \mathcal K_{D^{\bullet}}$ be the $\mbb A$-module with basis $\mathcal P_{D^{\bullet}}$. Let
$\mathcal K_{D^{\bullet}}$ be the $\mbb Q(v)$-vector space $\mbb Q(v) \otimes _{\mbb A} \, _{\mbb A} \mathcal K_{D^{\bullet}}$.
The functors  $\F_i^{(n)}$ and  $_i \mathcal R^{(n)}$ and $\mathcal R_i^{(n)}$ defined in Section ~\ref{special} descend to linear maps on
  $\mathcal K_{D^{\bullet}}$. We shall use the same letters to denote these linear maps.

From these linear maps, we can define a $\U$-module structure on $\mathcal K_{D^{\bullet}}$, and a $_{\mbb A}\! \U$-module structure on $_{\mbb A}\mathcal K_{D^{\bullet}}$. The functors $\F_i^{(n)}$ correspond  to the action $F_i^{(n)}$.
By (\ref{E}) and the fact that $_i\bar r$ and $\bar r_i$ gets identified with $_i \mathcal R$ and $\mathcal R_i$ in ~\cite{Lusztig93},  
the $E_i$-action on $\mathcal K_{D^{\bullet}}$  is defined by
\begin{equation}
\label{Ei}
E_i (x) =\frac{ \mathcal R_i (x) - v^{- (i, -|x| +i )} \, _i \mathcal R(x) } {v-v^{-1}}, \quad  \forall x\in \mathcal K_{D^{\bullet}} .
\end{equation}
It can also be expressed as 
\begin{equation}
\label{Ei-2}
E_i(x)= (\mbf D v^{(i,   |x|-i) } \,_i \mathcal R(\mbf D(x))
-v^{(i,   |x|-i) } \, _i \mathcal R(x))/(v-v^{-1}),
\quad  \forall x\in \mathcal K_{D^{\bullet}},
\end{equation}
where $\mbf D$ is the Verdier duality functor. 
The action  $K_i$  is corresponding to  the shift functors $P \mapsto P[ (i, \lambda-\nu)]$ in $\Q_{V, D^{\bullet}}$.
By (\ref{E}), we see that the module $\mathcal K_{D^{\bullet}}$ equipped with the defined actions is a $\U$-module.
Moreover, we have
\begin{thm}
\label{K-1}
There are isomorphisms 
\begin{equation}
\label{K-1-iso}
_{\mbb A}\! M_{\lambda^2} \otimes \, _{\mbb A}\! V_{\lambda^1} \simeq  \, _{\mbb A}\! \mathcal K_{D^{\bullet}}, \quad
M_{\lambda^2}\otimes V_{\lambda^1} \simeq \mathcal K_{D^{\bullet}},
\end{equation} 
as $_{\mbb A}\!\U$-modules and $\U$-modules, respectively. 
\end{thm}

\begin{proof}
By an argument similar to (\ref{psi'}), we have a surjective  homomorphism of $\U$-modules:
\[
M_{\lambda^2} \otimes V_{\lambda^1} \twoheadrightarrow \mathcal K_{D^{\bullet}}.
\]
By Theorem ~\ref{crystal-P}, the spaces $M_{\lambda^2} \otimes V_{\lambda^1}$ and $ \mathcal K_{D^{\bullet}}$ have the same dimension at each level. So the above morphism   is an isomorphism. The result over $\mbb A$ is proved in a similar way. Theorem follows.
\end{proof}

Under the isomorphism in the above Theorem, we see that $\mathcal P_{D^{\bullet}}$ is a basis of $M_{\lambda^2}\otimes V_{\lambda^1}$. We shall show in the following that it is the same as  the $canonical$ $basis$ of $M_{\lambda^2}\otimes V_{\lambda^1}$ defined in an algebraic way.

The involution $\mbf D: \mathcal K_{D^{\bullet}} \to \mathcal K_{D^{\bullet}}$ 
is compatible with the $\U$ action on $\mathcal K_{D^{\bullet}}$ in the sense that we have
\[
\mbf D(E_i x) = E_i (\mbf D(x)),\quad
\mbf D(F_i x) = F_i (\mbf D(x)),\quad
\mbf D(K_i x) =K_{-i} \mbf D(x), 
\] 
for any $i\in I$, $x\in \mathcal K_{D^{\bullet}}$.
Indeed, the compatibility of $\mbf D$ with $E_i$ follows from (\ref{Ei-2}) and the compatibility of $\mbf D$ with $F_i$ and $K_i$ is obvious. 

Let 
\[
\Psi: M_{\lambda^2}\otimes V_{\lambda^1}  \to M_{\lambda^2}\otimes V_{\lambda^1},
\]
be the unique involution defined in ~\cite[27.3.1]{Lusztig93} such that 
\[
\Psi(f x) =\bar f  \bar  x, \quad
\Psi(E_i x) = E_i \bar x,\quad 
\Psi(F_i x) =F_i  \bar x \quad \mbox{and} \quad 
\Psi(K_i x ) = K_{-i}\bar x, 
\]
for any  $f\in \mbb Q(v)$, $x\in M_{\lambda^2}\otimes V_{\lambda^1}$ and  $i\in I$, 
where $x\mapsto \bar x$ is the bar involution defined as the tensor product of bar involutions on $M_{\lambda^2}$ and $V_{\lambda^1}$, respectively.

\begin{cor}
\label{duality-P}
We have the following commutative diagram:
\[
\begin{CD}
M_{\lambda^2}\otimes V_{\lambda^1} @>>> \mathcal K_{D^{\bullet}}\\
@V\Psi VV @V \mbf D VV\\
M_{\lambda^2}\otimes V_{\lambda^1} @>>> \mathcal K_{D^{\bullet}},
\end{CD}
\]
where the horizontal morphisms are from (\ref{K-1-iso}).
\end{cor}

This can be proved by two steps. The first one is observed that the above diagram commutes for any elements of the form
$x\xi_{\lambda^1}\otimes \xi_{\lambda^2}$ in $M_{\lambda^2}\otimes V_{\lambda^1}$.  The second one is  that the diagram commutes for the rest  elements  can be proved by induction because the involutions $\Psi$ and $\mbf D$ are compatible with the $\U$-actions.

Recall from ~\cite[1.2.3]{Lusztig93}, we have a unique  symmetric bilinear form on $\U^-$ subject to the following properties:
\begin{align*} 
(1, 1) &= 1, \quad
(F_i, F_j)  = \delta_{ij} \frac{1}{1-v^{-2}}, \quad 
(F_i x, y) =  \frac{1}{1-v^{-2}} ( x, \;_i  r(y)),   
\end{align*}
for any $i, j\in I$ and $x, y\in \U^-$,
where $_ir: \U^-\to \U^-$ is a linear map defined by 
\[
_i r(1) =0, \; _ir(F_j) = \delta_{ij}, \; _i r(xy) =\, _i r(x) y + v^{|x|\cdot i} x \, _i r(y), \quad x, y \; \mbox{homogeneous}.
\]
As vector spaces, $\U^-$ is isomorphic to $M_{\lambda}$. So the bilinear form transports to $M_{\lambda}$.
Recall from ~\cite[19.1.1]{Lusztig93}, we have a unique symmetric bilinear form on $V_{\lambda}$ subject to the following properties:
\begin{align*}
(\xi_{\lambda}, \xi_{\lambda})=1, \quad 
(E_i x, y) = ( x, vK_i F_i y),\quad
(K_i x, y) =(x, K_iy), \quad 
(F_i x, y) = (x, vK_i^{-1} E_i y), \quad
\end{align*}
 for any $ x, y \in V_{\lambda}$ and $i\in I$.
 We define a bilinear form on $M_{\lambda^2}\otimes V_{\lambda^1}$ by  $(x_1\otimes x_2, y_1\otimes y_2) = (x_1,y_1) (x_2, y_2)$ for any $x_1, y_1\in M_{\lambda^2}$ and 
 $x_2, y_2\in V_{\lambda^1}$. 
 It is straightforward to show that the bilinear form on $M_{\lambda^2}\otimes V_{\lambda^1}$ satisfies the following properties:
 \begin{align*}
 (\xi_{\lambda^2}\otimes \xi_{\lambda^1}, \xi_{\lambda^2}\otimes \xi_{\lambda^1}) & =1,\\
 (\phi_i (x_1\otimes x_2), y_1\otimes y_2) &= \frac{1}{1-v^{-2}}  ( x_1\otimes x_2, \epsilon_i (y_1\otimes y_2)),    
 \end{align*}
 for any $x_1, y_1\in M_{\lambda^2}$, $x_2, y_2\in V_{\lambda^1}$, where $\phi_i$ and $\epsilon_i$, for  any $i\in I$, are defined by 
 \begin{align*}
  \phi_i(x_1\otimes x_2) &= F_i x_1 \otimes K_i^{-1} x_2 + x_1\otimes F_i x_2,\\
  \epsilon_i (x_1\otimes x_2) &= \, _i r(x_1) \otimes K_i^{-1} x_2 + (v-v^{-1}) x_1 \otimes K_i^{-1} E_i x_2, \quad \forall x_1 \otimes x_2\in M_{\lambda^2}\otimes V_{\lambda^1}.
 \end{align*}

Similar to ~\cite[13.1.6-13.1.12]{Lusztig93}, we can define geometrically a bilinear form on $\mathcal K_{D^{\bullet}}$:
\[
(, ): \mathcal K_{D^{\bullet}} \times \mathcal K_{D^{\bullet}}\to \mbb Q(v),
\]
subject to the following properties:
\begin{align*}
& (L_d, L_d) =1, \quad 
 (K, K') \in \delta_{K, K'} + v^{-1}\mbb Z[[v^{-1}]] \cap \mbb Q(v), \quad \forall K, K'\in \mathcal P_{D^{\bullet}},\\
& (F_i K, K') = \frac{1}{1-v^{-2}} (K,  \, _{i}\!\bar{\mathcal R} (K')), \quad \forall K, K'\in \mathcal P_{D^{\bullet}},
\end{align*}
where $_{i} \bar{\mathcal R} = \mbf D\circ \; _i\mathcal R \circ \mbf D$, which gets identified with the linear map $_ir$ on $\U^-_{\tilde \Gamma}$.

It is clear that the data $(\phi_i, \epsilon_i|i\in I)$ are compatible with the data $(F_i, \, _i \bar{\mathcal R}| i\in I)$ under the isomorphisms in Theorem ~\ref{K-1}. 
The above analysis shows that we have the lemma.

\begin{lem}
\label{bilinear}
The following diagram commutes:
\[
\begin{CD}
(M_{\lambda^2}\otimes V_{\lambda^1}) \otimes (M_{\lambda^2}\otimes V_{\lambda^1} )@> ( \, , \,  ) >> \mbb Q(v)\\
@VVV @|\\
\mathcal K_{D^{\bullet}} \times \mathcal K_{D^{\bullet}} @> ( \, ,\, ) >> \mbb Q(v),
\end{CD}
\]
where the vertical map is induced   from (\ref{K-1-iso}).
\end{lem}

Similar to ~\cite[Theorem 14.2.3]{Lusztig93}, we have 

\begin{thm}
\label{signed}
For any element $K\in \mathcal K_{D^{\bullet}}$, $\pm K\in \mathcal P_{D^{\bullet}}$ if and only if $K$ satisfies the following conditions:
\begin{align*}
K\in \; _{\mbb A} \! \mathcal K_{D^{\bullet}},\quad 
\mbf D(K) =K, \quad \mbox{and}\quad 
(K, K)\in 1+ v^{-1}\mbb Z[[v^{-1}]].
\end{align*}
\end{thm}

We can use the module $M_{\lambda^1}\otimes V_{\lambda^2}$ and its involution $\Psi$ to define the so-called $based$ $module$, $(M_{\lambda^1}\otimes V_{\lambda^2}, P_{\diamond})$ in an algebraic way
similar to  ~\cite[27.1.2, 27.3]{Lusztig93},  where we need to use the notion of crystal basis given in ~\cite[3.5]{K91} . 
The basis $P_{\diamond}$ is called  the $canonical$ $basis$ (or $global$ $crystal$ $base$) of $M_{\lambda^1}\otimes V_{\lambda^2}$.
Moreover,  by Corollary ~\ref{duality-P},  Lemma ~\ref{bilinear} and  Theorem ~\ref{signed},  we have 

\begin{thm}
\label{based-P}
The pair $(\mathcal K_{D^{\bullet}}, \mathcal P_{D^{\bullet}})$ together with the involution $\mbf D$  is isomorphic to the based module $(M_{\lambda^1}\otimes V_{\lambda^2}, P_{\diamond})$ with associated involution $\Psi$.
\end{thm}

\begin{rem}
(1). The two spaces $\mathcal K_{D^{\bullet}}$ and $\mbf K(\mbf d^{\bullet})$ are not the same in general, except when $\lambda^2=0$ or $\lambda^1=0$.

(2). It is very interesting to interpret naturally  the $E_i$-action as a complex of functors by using (\ref{Ei}) 
or the expression in Lemma ~\ref{qGLS-a}.

(3). The data $(M_{\lambda^2}\otimes V_{\lambda^1}, \phi_i, \epsilon_i)_{| i\in I}$   is a module of  Kashiwara's $reduced$ $q$-$analogue$ in ~\cite{K91}. 

(4). 
 Note that the modules $\mathcal K_{D^{\bullet}}$ are projective modules in  the Bernstein-Gelfand-Gelfand category $\mathcal O$  of the quantum group $\U$ in \cite{BGG76, H08, AM11}.   
We are very close to  recover all indecomposable projectives $P_{\lambda}$ with $\lambda \in \X$ in $\mathcal O$ from this geometric setting.  
\end{rem}

 \subsection{Stability conditions for $\Q_{V, D^{\bullet}}$}
 
Recall from Section ~\ref{tensorproductcomplexes}, $\N_{V, D^{\bullet}}$ is  the full subcategory of 
$\Q_{V, D^{\bullet}}$ generated by the simple objects $K$ satisfying the local stability condition (\ref{b}):
\[
\Supp (\Phi_{\Omega}^{\Omega_i}K) \subseteq   \E_{\Omega_i, i, \geq 1}(V, D), \quad \mbox{for some $i$ in $I$}.
\]
One can show that the condition is equivalent to the condition that the complex $K$ is a direct summand, up to a shift, of the complex  $ L_{(\mbf i^1,\mbf  a^1) d^1 (\mbf i^2, \mbf a^2) (i, d^2_i+1) d^2}$ for some $i\in I$.
 
 Similar to $\M_{V, D}$ in Section \ref{stable-1}, we define $\M_{V, D^{\bullet}}$ to be the full subcategory of $\Q_{V, D^{\bullet}}$ consisting of all complexes $K$ satisfying the 
 following global stability condition:
 \begin{equation}
 \label{tensor-global}
 \SSS(K) \cap \mbf  \Pi^s_{V, D^{\bullet}} =\mbox{\O}.
 \end{equation}

 \begin{thm}
 \label{main}
We have $\N_{V, D^{\bullet}}=\M_{V, D^{\bullet}}$.  
\end{thm}

\begin{proof}
The proof is similar to Proposition ~\ref{N=M}.
Suppose that $K$ is a simple object in $\N_{V, D^{\bullet}}$. 
Then we may assume that  $K$ is a direct summand of the complex $ L_{(\mbf i^1,\mbf  a^1) d^1 (\mbf i^2, \mbf a^2) (i, d^2_i+1) d^2}$ for some $i\in I$. By Proposition ~\ref{Tensor-singular-2}, 
$\SSS(K) \subseteq \tilde Y_{(\mbf i^1,\mbf  a^1) d^1 (\mbf i^2, \mbf a^2) (i, d^2_i+1) d^2}$ defined in Section ~\ref{singular}. 
Suppose that $X$ is an element in the latter variety.  Then it is clear that $X(i)$ is not injective. Thus
 $\tilde Y_{(\mbf i^1,\mbf  a^1) d^1 (\mbf i^2, \mbf a^2) (i, d^2_i+1) d^2}\cap \mbf \Pi_{V, D^{\bullet}}^s=\mbox{\O}$. 
 So we have  $K\in \M_{V, D^{\bullet}}$, i.e.,  $\N_{V, D^{\bullet}} \subseteq \M_{V, D^{\bullet}}$.
Moreover, the following sets have the same cardinality $\sum_{\nu^1+\nu^2=\nu}  \dim V_{\lambda^1, \nu^1}\otimes T_{\lambda^2, \nu^2} $:

the set, say $S_1$, of isomorphism classes of simple perverse sheaves in $\N_{V, D^{\bullet}}$;

the set, say $S_2$,  of irreducible components in $\mbf \Pi_{V, D^{\bullet}}$ disjoint from $\mbf \Pi_{V, D^{\bullet}}^s$.

Indeed, the fact that $\# S_1=\# S_2$,  can be proved in a similar way as that of $Y_{\bullet}$ is bijective in Proposition ~\ref{YY}  by taking consideration of 
the fact that if $Y\in S_2$, then $Y$ is in some subvariety of the form $\tilde Y_{(\mbf i^1,\mbf  a^1) d^1 (\mbf i^2, \mbf a^2) (i, d^2_i+1) d^2}$. 
And the fact that their total numbers of  elements are $\sum_{\nu^1+\nu^2=\nu}  \dim V_{\lambda^1, \nu^1}\otimes T_{\lambda^2, \nu^2} $ follows from the fact that
the total number of irreducible components of $\mbf \Pi_{V, D^{\bullet}}^s$ is $\sum_{\nu^1+\nu^2=\nu}  \dim V_{\lambda^1, \nu^1}\otimes V_{\lambda^2, \nu^2}$ 
by Propositions ~\ref{stable} and Theorem ~\ref{K-1}.

By Proposition  \ref{YY},  the assignment $K\mapsto Y_K$ defines a bijection $\phi: S_1\to S_2$.  By using Proposition ~\ref{YY} and  a similar argument in the proof of Proposition ~\ref{N=M}, 
we have  $\N_{V, D^{\bullet}}=\M_{V, D^{\bullet}}$. The theorem follows.
\end{proof}

Let $_{\mbb A}\mathcal L_{D^{\bullet}}$  and $\mathcal L_{D^{\bullet}}$ be the spaces defined similar to $_{\mbb A}\mathcal K_{D^{\bullet}}$ and $\mathcal K_{D^{\bullet}}$, respectively, if we replace $\Q_{V, D^{\bullet}}$ by its quotient category  $\V_{V, D^{\bullet}}$ in Section ~\ref{tensorproductcomplexes}.
By Theorems ~\ref{K-1} and ~\ref{main}, we have 

\begin{cor} 
\label{LD}
$_{\mbb A}\! V_{\lambda^2} \otimes \, _{\mbb A}\! V_{\lambda^1} \simeq  \, _{\mbb A}\! \mathcal L_{D^{\bullet}}$ and
$V_{\lambda^2}\otimes V_{\lambda^1} \simeq \mathcal L_{D^{\bullet}}.$
\end{cor}

Let $\pi_{D^{\bullet}}: \mathcal K_{D^{\bullet}} \to \mathcal L_{D^{\bullet}}$ be the natural projection. Let 
\[
\mathcal B_{D^{\bullet}}=\pi_{D^{\bullet}} (\mathcal P_{D^{\bullet}}) \backslash \{0\}.
\] 
It is clear by Theorems ~\ref{K-1} and ~\ref{main} that $\mathcal B_{D^{\bullet}}$ is a basis of $_{\mbb A}\! V_{\lambda^2} \otimes \, _{\mbb A}\! V_{\lambda^1}$ under the above isomorphism. 
Moreover, the involution $\mbf D$ induces an involution $\mbf D$ on $\mathcal L_{D^{\bullet}}$ commuting with the involution $\Psi$ on $V_{\lambda^2}\otimes V_{\lambda^1} $ defined in ~\cite[27.3.1]{Lusztig93} and  the crystal structure on $\mathcal P_{D^{\bullet}}$ induces a crystal structure on $\mathcal B_{D^{\bullet}}$.

\begin{cor} 
\label{based-V}
$(\mathcal L_{D^{\bullet}}, \mathcal B_{D^{\bullet}})$ with the associated involution $\mbf D$  is a based module in the sense of ~\cite[27.1.2]{Lusztig93}, isomorphic to $(V_{\lambda^2}\otimes V_{\lambda^1} , B_{\diamond})$, where $B_{\diamond}$ is the canonical basis of $V_{\lambda^2}\otimes V_{\lambda^1}$.
\end{cor}

The above analysis shows that the pair $(\mathcal L_{D^{\bullet}}, \mathcal B_{D^{\bullet}})$ satisfies the conditions ~\cite[27.1.2, (a)-(c)]{Lusztig93}. 
The condition ~\cite[27.1.2, (d)]{Lusztig93} can be proved by a result for $\mathcal P_{D^{\bullet}}$ analogous to ~\cite[Proposition 18.1.7]{Lusztig93}, which is left to the reader.
The identification of $\mathcal B_{D^{\bullet}}$ with $B_{\diamond}$ is due to a similar diagram as that in Corollary ~\ref{duality-P}.

\begin{rem}
Note that Corollaries ~\ref{LD} and  ~\ref{based-V} are first proved in ~\cite{Zheng08} by a different method.

The above results implies that   the second row in (\ref{M-tensor}) is a  categorical version of  the exact sequence
$0\to T_{\lambda^2} \otimes V_{\lambda^1}\to M_{\lambda^2}\otimes V_{\lambda^1} \to V_{\lambda^2}\otimes V_{\lambda^1} \to 0$.
\end{rem}

\section{General case}
\label{general}

The results in Section ~\ref{twocopy} can be generalized directly to the tensor product of $N$ copies of irreducible integrable representations.
We will state the analogous results  in this section. The results can be proved inductively  or by  mimicking the one in the  $N=2$ case.

\subsection{Results on $\mbf \Pi_{V, D^{\bullet}}$}
We fix  $\lambda^1$, $\cdots$, $\lambda^N$ in $\X$,  $d^1, \cdots, d^N$ in $\mbb N[I]$ and $D^1, \cdots, D^N$ such that
\[
\lambda =\lambda^1+\cdots+\lambda^N, \quad d=d^1+\cdots +d^N,\quad \mbox{and}\quad D=D^N\oplus \cdots \oplus D^1,
\]
and  $ d_i^a =\lambda^a(i) =\dim D^a_i$,  for any $i\in I$ and $a=1,\cdots, N$. Let 
\[
D^{\bullet} = (\check D^0 \supset \check D^1 \supset \cdots \supset \check D^N) 
\]
be the flag with $\check D^a = \oplus_{b=a+1}^N D^b$ for $a=0,\cdots, N$. 
The tensor product variety $\mbf \Pi_{V, D^{\bullet}}$ is the closed subvariety of 
$\mbf \Lambda_{V, D}$ consisting of all nilpotent elements $X=(x, p, q)$ such that 
 \begin{equation}
  \overline{p(\check D^a)} \subseteq \underline{q^{-1} (\check D^{a+1})}\quad \forall a=0,\cdots, N-2
 \quad \mbox{and}\quad 
 p(\check D^{N-1})=0.
 \end{equation}
When $N=2$, this is the same as the one defined in (\ref{tensor-condition}).

One can prove inductively the following results.

\begin{thm} 
\label{property-rho-N}
The following statements hold.
\begin{enumerate}
\item $\mbf \Pi_{V, D^{\bullet}}$  has   pure dimension  $\frac{1}{2} \dim \E(V, D)$. Moreover, 
         \[\# \Irr \mbf \Pi_{V, D^{\bullet}} =\sum_{V^1, \cdots, V^N} \# \Irr ( \mbf   L^s_{V^1,D^1} \times \cdots \times \mbf  L^s_{V^{N-1}, D^{N-1}}  \times \mbf  L_{V^N,D^N}),\] 
         where the sum runs over the representatives $(V^1, \cdots,  V^N)$ of $(\nu^1, \cdots, \nu^N)$ such that $\nu^1+\cdots+\nu^N=\nu=\dim V$.
\item The set $\Irr (D^{\bullet})=\sqcup_V \Irr \mbf \Pi_{V, D^{\bullet}}$, equipped with a crystal structure defined similar  to that  in Section ~\ref{IrrD}, is isomorphic to the crystal
           $B(\lambda^1)\otimes \cdots \otimes B(\lambda^{N-1})        \otimes B(\lambda^N,\infty)$. 
\end{enumerate}
\end{thm}

Note that if we consider the class of constructible functions on $\mbf \Pi_{V, D^{\bullet}}$,  similar to the complexes $L_{\underline{\mbf a}}$, 
we will get a geometric realization of the module  of the enveloping algebra $U$ associated with the graph $\Gamma$, similar to the module 
$M_{\lambda^N}\otimes V_{\lambda^{N-1}}\otimes \cdots \otimes V_{\lambda^1}$.  The $E_i$-action can be defined in a way similar to that of $M_{\lambda}$  in ~\cite{GLS06}.  A basis for this $U$-module can also be obtained as the semicanonical basis  for $U^-$ in ~\cite{L00a}.

\subsection{Results on $\mathcal P_{V, D^{\bullet}}$}
Let 
\[
\underline{\mbf a} := (\mbf i^1, \mbf a^1) \cdot d^1\cdots (\mbf i^{N-1}, \mbf a^{N-1})  \cdot d^{N-1} \cdot  (\mbf i^N, \mbf a^N) \cdot d^N.
\]
We can define the object $\Q_{V, D^{\bullet}}$, $\mathcal P_{V, D^{\bullet}}$,  $\mathcal P_{D^{\bullet}}$, $\mathcal N_{V, D^{\bullet}}$, $\mathcal M_{V, D^{\bullet}}$, $\mathcal V_{V, D^{\bullet}}$ 
$\mathcal B_{D^{\bullet}}$ and $\mathcal K_{D^{\bullet}}$ in exactly the same manner as those in Section ~\ref{twocopy}.
Moreover, the results in Section ~\ref{twocopy} are still true with  minor modifications.
The proofs are similar to those in $N=2$ cases by taking into consideration of Proposition ~\ref{property-rho-N}. In particular, we have 

\begin{thm}
\begin{enumerate}
\item The set $\mathcal P_{D^{\bullet}}$, together with the crystal structure defined similar to the one in Section ~\ref{mapY},  is isomorphic to the crystal $B(\lambda^1)\otimes \cdots \otimes B(\lambda^{N-1})        \otimes B(\lambda^N,\infty)$. 

\item For any $K\in \mathcal P_{V D^{\bullet}}$, its singular support satisfies a similar relation as in Proposition \ref{YY}.

\item $(\mathcal K_{D^{\bullet}}, \mathcal P_{D^{\bullet}})$ is isomorphic to the based  module $(M_{\lambda^N }\otimes V_{\lambda^{N-1}} \otimes  V_{\lambda^{N-2}} \otimes \cdots     \otimes V_{\lambda^1}, P_{\diamond}) $, 
          where $P_{\diamond}$ is the  canonical basis of $M_{\lambda^N }\otimes V_{\lambda^{N-1}} \otimes  V_{\lambda^{N-2}} \otimes  \cdots \otimes  V_{\lambda^1}$. 

\item $\N_{V,D^{\bullet}}=\M_{V, D^{\bullet}}$.
\item The based module $(\mathcal L_{D^{\bullet}}, \mathcal B_{D^{\bullet}})$ is isomorphic to $(V_{\lambda^{\bullet} }, B_{\diamond})$ where $B_{\diamond}$ is the canonical basis of $V_{\lambda^{\bullet}}$.
\end{enumerate}
\end{thm}

Note that (5) has been proved in ~\cite{Zheng08} by a different method.

\section{Acknowledgement}
We thank Catharina Stroppel for her invitation to University of Bonn, where this work was originated and sending us the paper ~\cite{FSS11}. 
We thank Wei Liang Gan for his invitation to present this work in the workshop of Lie groups, Lie algebras and their representations, University of California, Riverside, May 21-22, 2011.
We  also thank  Igor B. Frenkel, Zongzhu  Lin, Raphael Rouquier,  Mark  Shimozono,
Josh Sussan and Ben Webster for helpful conversations. We thank Jim Humphreys for sending the author a copy of ~\cite{H77}. 
We thank Yoshiyuki Kimura for pointing out a mistake in a previous version of this paper and sending us the paper ~\cite{K07}.

Partial support from NSF grant DMS 1101375 / 1160351  is  acknowledged. Part of the work was done while the author was in Virginia Tech.

\end{document}